\documentclass[a4paper,10pt, reqno]{amsart}

\usepackage{amsmath,amssymb,latexsym, amsfonts}
\usepackage{verbatim}
\usepackage{times}
 \usepackage{mathbbol}
\usepackage{hyperref}
 \usepackage[usenames]{color}
 \usepackage{mathabx}
 \usepackage{bbm}
\usepackage[all]{xy}
 \usepackage{stmaryrd}
\usepackage{graphicx}
\usepackage{mathrsfs}

\newcommand{\compactlist}[1]{\setlength{\itemsep}{0pt} \setlength{\parskip}{0pt} \setlength{\leftskip}{-0.#1em}}

\numberwithin{equation}{section}

\theoremstyle{plain}
\newtheorem{theorem}{Theorem}[subsection]
\newtheorem{proposition}[theorem]{Proposition}
\newtheorem{prop}[theorem]{Proposition}
\newtheorem{lemma}[theorem]{Lemma}

\theoremstyle{definition}
\newtheorem{definition}[theorem]{Definition}
\newtheorem{example}[theorem]{Example}

\newtheorem{remark}[theorem]{Remark}

\newtheorem{note}[theorem]{Note}
\newtheorem{free text}[theorem]{}

\newcommand{\ahha}{{\scriptscriptstyle{A}}}
\newcommand{\behhe}{{\scriptscriptstyle{B}}}

\newcommand{\emme}{{\scriptscriptstyle{M}}}
\newcommand{\enne}{{\scriptscriptstyle{N}}}

\newcommand{\uhhu}{{\scriptscriptstyle{U}}}

\newcommand{\N}{{\mathbb{N}}}

\newcommand{\Z}{{\mathbb{Z}}}

\newcommand{\ga}{\alpha}

\newcommand{\gD}{\Delta}
\newcommand{\eps}{\epsilon}
\newcommand{\gve}{\varepsilon}

\newcommand{\gl}{\lambda}

\newcommand{\cD}{{\mathcal D}}

\newcommand{\cP}{{\mathcal P}}

\newcommand{\Hom}{\operatorname{Hom}}
\newcommand{\Der}{\operatorname{Der}}

\newcommand{\Ext}{{\rm Ext}}

\newcommand{\id}{{\rm id}}

\newcommand{\Ker}{{\rm Ker}\,}

\newcommand{\due}[3]{{}_{{#2 }} {#1}_{{ #3}}\,}    % Zweifachindex
\newcommand{\qttr}[5]{{}^{{#2 \!}}_{{#4 \!}} {#1}^{#3}_{{\! #5}}}    % Vierfachindex
\newcommand{\qttrd}[5]{{}^{{#2 \!\!}}_{{#4}} {#1}^{\!\! #3}_{{#5}}}    % Vierfachindex mit speziellen Einrueckungen

\newcommand{\pl}{\partial}

\newcommand{\rmref}[1]{{\text (}\ref{#1}{\text )}}

\newcommand{{\Hl}}{{H^{\ell}}}
\newcommand{{\mHop}}{{m_{H^{\rm op}}}}
\newcommand{{\Hop}}{{H^{\rm op}}}
\newcommand{{\mUop}}{{m_{U^{\rm op}}}}
\newcommand{{\mUopp}}{{m_{\scriptscriptstyle{U^{\rm op}}}}}
\newcommand{{\Uop}}{{U^{\rm op}}}
\newcommand{{\mVop}}{{m_{V^{\rm op}}}}
\newcommand{{\Vop}}{{V^{\rm op}}}
\newcommand{{\Ae}}{{A^{\rm e}}}
\newcommand{{\Ue}}{{U^{\rm e}}}
\newcommand{{\He}}{{H^{\rm e}}}
\newcommand{{\Aop}}{{A^{\rm op}}}
\newcommand{{\Aope}}{({A^{\rm op}})^{\rm e}}
\newcommand{{\Aopl}}{{A^{\rm op}_\pl}}
\newcommand{{\Bop}}{{B^{\rm op}}}
\newcommand{{\Bope}}{({B^{\rm op}})^{\rm e}}
\newcommand{{\Bpl}}{{B_\pl}}
\newcommand{{\op}}{{{\rm op}}}
\newcommand{{\coop}}{{{\rm coop}}}
\newcommand{{\sop}}{{*^{\rm op}}}
\newcommand{\amod}{A\mbox{-}\mathbf{Mod}}                     %
\newcommand{\amoda}{A^{\rm e}\mbox{-}\mathbf{Mod}}                  %
\newcommand{\umod}{U\mbox{-}\mathbf{Mod}}                     %  Modul-Kategorien
\newcommand{\modu}{\mathbf{Mod}\mbox{-}U}         %
\newcommand{\comodu}{\mathbf{Comod}\mbox{-}U}
\newcommand{\ucomod}{U\mbox{-}\mathbf{Comod}}

\newcommand{\ucmod}{U_h\mbox{-}\mathbf{cMod}}
\newcommand{\cmodu}{\mathbf{cMod}\mbox{-}U_h}

 \newcommand{\lact}{\smalltriangleright}
 \newcommand{\ract}{\smalltriangleleft}
 \newcommand{\blact}{\blacktriangleright}
 \newcommand{\bract}{\blacktriangleleft}

\newcommand{{\gog}}{{G \rightrightarrows G_0}}

\newcommand{{\rra}}{\rightrightarrows}
\newcommand{{\lra}}{\ \longrightarrow \ }
\newcommand{{\lla}}{\ \longleftarrow \ }
\newcommand{{\lma}}{\ \longmapsto \ }

\newcommand{{\bull}}{{\scriptscriptstyle{\bullet}}}
\newcommand{{\qqquad}}{{\quad\quad\quad}}
\newcommand{\Aopp}{{\scriptscriptstyle{\Aop}}}
\newcommand{\Bopp}{{\scriptscriptstyle{\Bop}}}

\newcommand{\scalast}{{\raisebox{-.7mm}{\scalebox{0.8}{*}}}}
\newcommand{\scalastd}{{\raisebox{-.4mm}{\scalebox{0.8}{*}}}}

\begin{document}

\title{Duality features of left Hopf algebroids}

\author{Sophie Chemla}
\author{Fabio Gavarini}
\author{Niels Kowalzig}

\begin{abstract}
We explore special features of the pair $(U^{\scalast}, U_{\scalastd})$
formed by the right and left dual over a (left) bialgebroid $U$
in case the bialgebroid  is, in particular, a left Hopf algebroid. It turns out that there exists a bialgebroid morphism  $ S^{\scalast} $
from %$U^{\scalast}$
one dual
to %$U_{\scalastd} $
another
that extends the construction of the antipode on the dual of a Hopf algebra, and which is an isomorphism if $U$ is both a left and right Hopf algebroid.
% This structure will be used to recover
This structure is derived from Ph\`ung's categorical equivalence between left and right comodules over $U$ without the need of a (Hopf algebroid) antipode, a result which we review and extend.
In the applications, we illustrate the difference between this construction and those involving antipodes and also deal with dualising modules and their quantisations.
\end{abstract}

\address{S.C.: Institut de Math\'ematiques de Jussieu,  UMR 7586, Universit\'e Pierre et Marie Curie, 75005 Paris, France}
\email{sophie.chemla@imj-prg.fr}

\address{F.G.: Dipartimento di Matematica, Universit\`a di Roma Tor Vergata,
Via della Ricerca Scientifica~1, 00133 Roma, Italia}
\email{gavarini@mat.uniroma2.it}

%\address{N.K.: Universit\`a di Napoli Federico II, Dipartimento di Matematica e Applicazioni,
%P.le Tecchio 80, 80125 Napoli, Italia}
%\email{niels.kowalzig@unina.it}

\address{N.K.: Istituto Nazionale di Alta Matematica, P.le Aldo Moro 5, 00185 Roma, Italia}
\email{kowalzig@mat.uniroma1.it}



\maketitle

\tableofcontents

\section{Introduction}

A characteristic feature in standard Hopf algebra theory is its self-duality, that is, the dual of a (finite-dimensional) Hopf algebra (over a field) is a Hopf algebra again. In particular, the antipode of this dual  is nothing but the transpose of the original antipode; see, for example, \cite{Swe:HA}.
In the broader setup of ({\em left} or {\em full}) Hopf algebroids over possibly noncommutative rings, this peculiar property appears to be more intricate; see \cite{Boe:HA} or \S\ref{tranquilli} for the precise definitions of these objects, we only mention here that, in contrast to full Hopf algebroids, there is no notion of antipode for left Hopf algebroids:
one rather considers the inverse of a certain Hopf-Galois map and its associated {\em translation map}.
Nevertheless, left Hopf algebroids appear as the correct generalisation of Hopf algebras over noncommutative rings, whereas full Hopf algebroids generalise Hopf algebras twisted by a character, see, for example, \cite[\S4.1.2]{Kow:HAATCT}.

Recently (after the first posting of this article), Schauenburg  \cite{Schau:TDATDOAHAAHA} showed that the (skew) dual of a left Hopf algebroid (under a suitable finiteness assumption) carries some Hopf structure as well without giving an explicit expression for the inverse of the respective Hopf-Galois map or the associated translation map.

However,
instead of one dual, a left bialgebroid  $ U $ rather possesses {\em two},  the  {\em right dual\/}  $ U^{\scalast} $  and the {\em left dual\/}  $ U_{\scalastd}  $, which, on top, live in a different category compared to $U$ as they are both (under certain finiteness assumptions) right bialgebroids \cite{KadSzl:BAODTEAD}.
There is no reason why one should prefer one of the duals to the other. Hence,  any question concerning ``the dual of  $ U $''  should be converted into a question about the pair  $ ( U^{\scalast} , U_{\scalastd} ) $.

Dealing with {\em full} Hopf algebroids (see \S\ref{reminders_H-ads}) does notably worsen the situation as there are actually {\em four} duals to be taken into account, two of which are left and two of which are right bialgebroids. In this case, an answer to the question of the nature of the Hopf structure on the dual(s) has only been given in certain cases, more precisely, in the presence of integrals \cite[\S5]{BoeSzl:HAWBAAIAD}.

\subsection{Aims and objectives}
\label{aims,appunto}
As mentioned a moment ago, the object one should investigate to discover the limits of self-duality in (left) Hopf algebroid theory is a {\em pair} of duals. In short, our question reads as follows: if a left bialgebroid  $ U $ is, in particular, a left (or right) Hopf algebroid,  what extra structure can be found on the pair
$ ( U^{\scalast} , U_{\scalastd} ) $  of duals?

\subsection{Main results}

After highlighting in \S\ref{acquaprimavera} a multitude of module structures that exist on $\Hom$-spaces and tensor products in presence of a left or right Hopf algebroid structure and that will be used in the sequel, in \S\ref{fallouts} we review (and extend) Ph\`ung's equivalence of comodule categories (see the main text for all definitions and conventions used hereafter):

{\renewcommand{\thetheorem}{{A}}
%\usecounter{theorem}
\begin{theorem}
Let $(U,A)$ be a left bialgebroid.
\begin{enumerate}
\compactlist{99}
\item
Let $(U,A)$ be additionally a left Hopf algebroid such that $U_{\ract}$ is projective.
Then there exists a (strict) monoidal functor  $  \comodu \to \ucomod  $: if $M$ is a right $U$-comodule with coaction $m \mapsto m_{(0)} \otimes_\ahha m_{(1)}$, then
$$M \to U_\ract \otimes_\ahha M, \quad m \mapsto  m_{(1)-} \otimes_\ahha  m_{(0)} \epsilon(m_{(1)+}),$$
defines a left comodule structure on $M$ over $U$.
\item
Let  $(U,A)$  be a right Hopf algebroid such that $_{\lact}U$ is projective.
 Then there exists a (strict) monoidal functor  $  \ucomod \to \comodu  $: if $N$ is a left $U$-comodule with coaction $n \mapsto n_{(-1)} \otimes_\ahha n_{(0)}$, then
$$N \to N \otimes_\ahha \due U \lact {}, \quad n \mapsto \epsilon(n_{(-1)[+]}) n_{(0)} \otimes_\ahha  n_{(-1)[-]},$$
defines a right comodule structure on $N$ over $U$.
\item
If $U$ is both a left and right Hopf algebroid and if both $U_{\ract}$ and $_{\lact}U$ are $A$-projective,
then the functors mentioned in {\it (i)} and {\it (ii)} are quasi-inverse to each other and we have an equivalence
$$
\ucomod \simeq \comodu$$
of monoidal categories.
\end{enumerate}
\end{theorem}
}

Note that this equivalence works without the help of an antipode as there are objects that are both left and right Hopf algebroids but not full Hopf algebroids (cocommutative left Hopf algebroids, for example).

Starting from this result, under suitable finiteness hypotheses on $U$, one can construct functors $\mathbf{Mod}\mbox{-}{U_{\scalastd}} \to \mathbf{Mod}\mbox{-}{U^{\scalastd}}$ resp.\ $\mathbf{Mod}\mbox{-}{U^{\scalastd}} \to \mathbf{Mod}\mbox{-}{U_{\scalastd}}$, and from this we isolate maps $U^{\scalastd} \to U_{\scalastd}$ resp.\ $U_{\scalastd} \to U^{\scalastd}$, which even make sense without any finiteness assumptions as proven in \S\ref{linking-structures},
and which are our main object of interest.

In \S\ref{linking} we can then give the following answer to the problem mentioned in \S\ref{aims,appunto}, that is, elucidate the relation between the left and the right dual:

{\renewcommand{\thetheorem}{{B}}
%\usecounter{theorem}

\begin{theorem}
\label{acquadinepi}
Let $(U,A)$ be a left bialgebroid.
\begin{enumerate}
\compactlist{99}
\item
If  $(U,A) $  is moreover a left Hopf algebroid, there is a morphism $  S^{\scalast} : U^{\scalast} \to U_{\scalastd}  $
of  $ A^e $-rings  with augmentation; if, in addition, both  $ \due U \lact {} $  and  $ U_\ract $  are finitely generated $A$-projective,  then  $ (S^{\scalast}, \id_\ahha) $  is a morphism of right bialgebroids.
\item
If  $(U,A) $  is a right Hopf algebroid instead, there is a morphism
$  S_{\scalastd} : U_{\scalastd} \to U^{\scalast}  $
of  $ A^e $-rings  with augmentation; if, in addition, both  $ \due U \lact {} $  and  $ U_\ract $  are finitely generated $A$-projective,  then  $ (S_{\scalast}, \id_\ahha) $  is a morphism of right bialgebroids.
\item
If $(U,A)$ is simultaneously both a left and a right Hopf algebroid, then the two morphisms  are inverse to each other; hence, if both  $ \due U \lact {} $  and  $ U_\ract $  are finitely generated $A$-projective, then
$U^\scalast\simeq U_{\scalastd}$ as right bialgebroids.
\end{enumerate}
\end{theorem}
}

Now, as said before, for a left Hopf algebroid (which is finitely generated projective with respect to both source and target map) there is no canonical choice for which dual to consider but in view of Theorem \ref{acquadinepi}, in case the left Hopf algebroid is simultaneously a right Hopf algebroid, both duals are isomorphic and hence can be seen as {\em its} dual, which carries a Hopf structure by Schauenburg's recent result \cite{Schau:TDATDOAHAAHA}. This seems to be as close as one can get to self-duality.

Theorem \ref{acquadinepi} is a straight analogue of the construction on the dual for a (finite-dimensional) Hopf algebra $H$ (over a field) with antipode $S$ in the following sense: here, one has  $  H^{\scalast} = (H_{\scalastd})^\op_\coop  $
and  $ S^{\scalast} $  is exactly the transpose of $S$ and therefore the antipode for the dual Hopf algebra.

   Observe that this last case in Theorem \ref{acquadinepi}, {\em i.e.}, the presence of both a left and right Hopf structure is given, for example, when  $ U $  is a full Hopf algebroid with bijective antipode but also in weaker cases such as for the universal enveloping algebra of a Lie-Rinehart algebra. In the situation of a full Hopf algebroid, $U^{\scalast}$ and $U_{\scalastd}$ are additionally linked (in both directions) by the transposition $\qttr S t {}{}{}$ of the antipode $  S : U \to U^\op_\coop  $. However, in Theorem \ref{movimentoallaforza} we show that the map $\qttr S t {}{}{}$ in general does not coincide with $S^{\scalast}$ or $S_{\scalastd}$, in contrast to the Hopf algebra case mentioned above.
Moreover, if  a left Hopf algebroid $ U $  is cocommutative with both  $ \due U \lact {} $  and  $ U_\ract $  finitely generated $A$-projective,
then $  U^{\scalast}  = (U_{\scalastd})_\coop  $  is a full Hopf algebroid (with antipode precisely given by $S^{\scalast}$), though  $ U $  might be not.

   We shall also see in \S\ref{lautesGequatsche} that  Theorem \ref{acquadinepi}  actually extends to a larger setup, in particular,  it applies to some interesting cases  (coming from geometry), where neither  $ \due U \lact {} $  nor  $ U_\ract $  are finitely generated projective but $U^\scalast$ and $U_\scalast$ are still right bialgebroids in a suitable (topological) sense,  such as when  $ U $  is the universal enveloping of a Lie-Rinehart algebra, or a quantisation of it.

In \S\ref{lautesGequatsche},
we illustrate these results by considering some examples related to Lie-Rinehart algebras (or Lie algebroids) and their jet spaces, as well as their quantised versions.
Moreover, in \S\ref{heisz} we consider further duality phenomena related to dualising modules, which appear in Poincar\'e duality, along with their quantisations.

\bigskip

\thanks{ {\bf Acknowledgements.}
 We would like to thank the referee for useful comments and suggestions, and in particular for pointing out the article \cite{Phu:TKDFHA} to us.
We also would like to thank G.\ B{\"o}hm and L.\ El Kaoutit for their valuable comments and remarks.

   S.C.\ acknowledges travel support and hospitality from both INdAM and Universit\`a di Roma Tor Vergata, where part of this work was achieved.
The research of F.G.\ was partially funded by the PRIN 2012
    project {\em Moduli spaces and Lie theory}.
The research of N.K.\ was funded by an INdAM-COFUND Marie Curie grant and also supported by UniNA and Compagnia di San Paolo in the framework of the program STAR 2013.

\section{Preliminaries}
\label{tranquilli}

We list here those preliminaries with respect to bialgebroids and their duals that are needed in this article; see, {\em e.g.}, \cite{Boe:HA} and references therein for an overview on this subject.

Fix an (associative, unital, commutative) ground ring  $ k $.  Unadorned tensor products will always be meant over  $ k $.
All other algebras, modules etc.~will
have an underlying structure of a
$k$-module. Secondly, fix an associative and unital
$k$-algebra $A$, {\em i.e.}, a ring with a
ring homomorphism
$ \eta_\ahha : k \rightarrow Z(A)$ to
its centre. Denote by
$A^\mathrm{op}$ the
opposite and by
$A^\mathrm{e} := A \otimes A^\mathrm{op}$
the enveloping algebra
of $A$, and by
$\amod$ the category of
left $A$-modules.
 Recall that an {\em $ A $-ring\/}  is a monoid in the monoidal category $(\amoda, \otimes_\ahha, A)$ of $(A,A)$-bimodules fulfilling the usual associativity and unitality axioms, whereas dually an {\em $A$-coring\/} is a comonoid in this category that is coassociative and counital.

\subsection{Bialgebroids}
\label{h-Hopf_algbds}

For an  $ A^e $-ring $ U $  given by the $k$-algebra map
  $\eta: \Ae \to U $,  consider the restrictions
$ s := \eta( - \otimes 1_\uhhu) $  and  $t:= \eta(1_\uhhu \otimes -)$, called  {\it source\/}  and  {\it target\/}  map, respectively.
Thus an  $ A^e $-ring  $ U $ carries two  $ A $-module  structures from the left and two from the right, namely
  $$
a \lact u \ract b  :=  s(a)  t(b)  u,  \quad   \quad  a \blact u \bract b  :=  u  t(a)  s(b),   \eqno \forall \; a, b \in A  ,  u \in U.
$$
If we let  $  U_\ract {\otimes_{\scriptscriptstyle A}} {}_\lact U  $  be the corresponding tensor product of  $ U $  (as an $ A^e $-module)  with itself, we define the  {\it (left) Takeuchi-Sweedler product\/}  as
$$
U_\ract \! \times_\ahha \! {}_\lact U  \; := \;
     \big\{ {\textstyle \sum_i} u_i \otimes u'_i \in U_\ract \! \otimes_{\scriptscriptstyle A} \! {}_\lact U \mid {\textstyle \sum_i} (a \blact u_i) \otimes u'_i = {\textstyle \sum_i} u_i \otimes (u'_i \bract a), \ \forall a \in A \big\}.
$$
By construction,  $  U_\ract \! \times_{\scriptscriptstyle A} \! {}_\lact U  $  is an  $ \Ae $-submodule  of  $  U_\ract \! \otimes_{\scriptscriptstyle A} \! {}_\lact U  $;  it is also an  $ A^e $-ring via factorwise multiplication, with unit $  1_\uhhu \otimes 1_\uhhu  $  and  $ \eta_{{}_{U_\ract \times_{\scriptscriptstyle A} {}_\lact U}}(a \otimes \tilde{a}) := s(a) \otimes t(\tilde{a})$.

Symmetrically, one can consider the tensor product
$  U_\bract \otimes_\ahha \due U \blact {} $  and define the  {\em (right) Takeuchi-Sweedler product\/}  as
$U_\bract \times_\ahha \due U \blact {}  $,   which is an  $\Ae $-ring  inside
$  U_\bract \otimes_\ahha \due U \blact {} $.

\begin{definition}
 A {\em left  bialgebroid} $(U,A)$  is a  $ k $-module  $ U $  with the structure of an
$ \Ae $-ring  $(U, s^\ell, t^\ell)$  and an  $ A $-coring  $(U, \gD_\ell, \eps)$  subject to the following compatibility relations:
\begin{enumerate}
\item
the  $ \Ae $-module  structure on the  $ A $-coring  $ U $  is that of
$ \due U \lact \ract  $;
\item
the coproduct $ \Delta_\ell $  is a unital  $ k $-algebra  morphism taking values in  $  U {}_\ract \! \times_{\scriptscriptstyle A} \! {}_\lact U  $;
\item
for all  $  a, b \in A  $,  $  u, u' \in U  $, one has:
\begin{equation}
\label{castelnuovo}
\epsilon( a \lact u \ract b) =  a  \epsilon(u)  b, \quad \epsilon(uu')  =  \epsilon \big( u \bract \epsilon(u')\big) =  \epsilon \big(\epsilon(u') \blact u\big).
\end{equation}
\end{enumerate}
A  {\it morphism\/}  between left bialgebroids $(U, A)$ and $(U',A')$
is a pair $(F, f)$ of maps $F: U \to U'$, $f:A \to A'$ that commute with all structure maps in an obvious way.
\end{definition}

As for any ring, we can define the categories $\umod$ and $\modu$ of left and right modules over $U$. Note that $\umod$ forms a monoidal category but $\modu$ usually does not. However, in both cases there is a forgetful functor $\umod \to \amoda$, resp.\ $\modu \to \amoda$: whereas we denote left and right action of a bialgebroid $U$ on $M \in \umod$ or $N \in \modu$ usually by juxtaposition, for the resulting $\Ae$-module structures the notation
$$
a \lact m \ract b := s^\ell(a)t^\ell(b)m, \qquad a \blact m \bract b := ns^\ell(b)t^\ell(a)  %, \qquad m \in M, \ n \in N, \ a,b \in A,
$$
for $m \in M, \ n \in N, \ a,b \in A$ is used instead. For example, the base algebra $A$ itself is a left $U$-module via the left action
$u(a) := \epsilon( u \bract a) = \epsilon( a \blact u )$ for  $  u \in U  $ and  $  a \in A  $, but in most cases there is no right $U$-action on $A$.

Dually, one can introduce the categories $\ucomod$ and $\comodu$ of left resp.\ right $U$-comodules, both of which are monoidal; here again, one has forgetful functors $\ucomod \to \amoda$ and $\comodu \to \amoda$.
More precisely (see, {\em e.g.}, \cite{Boe:HA}), a (say) right comodule
is a right comodule of the coring underlying $U$, {\em i.e.}, a right $A$-module $M$ and a right $A$-module map
$
     \due {\gD} {\emme\!} {}: M \rightarrow
        M \otimes_\ahha {}_\lact  U, \
       m \mapsto m_{(0)} \otimes_\ahha m_{(1)},
$
satisfying the usual coassociativity and counitality axioms.
On any $M \in \comodu$ there is an induced {\em left} $A$-action given by
\begin{equation}
\label{Inducedaction}
am := m_{(0)}\gve(a \blact m_{(1)}),
\end{equation}
and $\due {\gD} {\emme\!} {}$ is then an $\Ae$-module morphism
$
M \rightarrow
        M \times_\ahha {}_\lact  U,
$
where $ M \times_\ahha {}_\lact  U $ is the $\Ae$-submodule of
$ M \otimes_\ahha {}_\lact  U$ whose elements $\sum_i m_i
\otimes_\ahha u_i$ fulfil
\begin{equation}
\label{Takeuchicoaction}
\textstyle
\sum_i a m_i \otimes_\ahha u_i
		  = \sum_i m_i \otimes_\ahha u_i \bract a, \
		  \forall a \in A.
\end{equation}

\medskip

The notion of a  {\em right bialgebroid\/}  is obtained if one starts with the $\Ae$-module structure given by $ \blact  $ and  $ \bract  $ instead of $\lact$ and $\ract$. We will refrain from giving the details here and refer to \cite{KadSzl:BAODTEAD} instead.

\begin{remark}
 The  {\em opposite\/}  of a left  bialgebroid  $( U, A, s^\ell, t^\ell, \Delta_\ell, \epsilon )  $  yields  a {\em right\/}  bialgebroid
$( U^\op, A, t^\ell, s^\ell, \Delta_\ell, \epsilon )$.
The  {\em coopposite\/}  of a left bialgebroid is the  {\em left\/}  bialgebroid given by  $( U, A^\op, t^\ell, s^\ell, \Delta_\ell^\coop, \epsilon) $.
\end{remark}

\subsection{Pairings of  $ U $-modules  and dual bialgebroids}
\label{mayitbe}
Let  $ (U,A) $  be a left bialgebroid, $M, M' \in \umod$ be left $U$-modules, and $N,N' \in \modu$ be right $U$-modules. Define
\begin{equation*}
\begin{array}{rclrcl}
\Hom_{\Aopp} (M , M')  \!\!\!&:=\!\!\!&  \Hom_{\Aopp} ( M_\ract ,  M'_\ract ), &
\Hom_\ahha( M , M')  \!\!\!&:=\!\!\!&  \Hom_\ahha (\due M \lact {}, \due {M'} \lact {}),  \\
\Hom_{\Aopp} (N , N') \!\!\!&:=\!\!\!&  \Hom_{\Aopp} ( N_{\bract}  , N'_{\bract}), &
\Hom_\ahha(N , N')  \!\!\!&:=\!\!\!&  \Hom_\ahha(\due N \blact {}, \due {N'} \blact {}).
\end{array}
\end{equation*}
%with respect to the induced  $ \Ae $-module  structures in either case.
In particular,  for  $ M': = A  $  we set  $  M_{\scalastd} := \Hom_\ahha( M , A )  $  and  $ M^{\scalast} := \Hom_{\Aopp}( M , A)  $,  called, respectively, the  {\em left\/}  and {\it right\/}  dual of  $ M  $.

   The notion of  {\em pairing\/}  between $\Ae$-bimodules is also useful (see, for instance, \cite{CheGav:DFFQG}):

\begin{definition}
\label{left and right 1}
Let  $ U $  and  $ W $  be two  $ \Ae $-bimodules.
\begin{enumerate}
\compactlist{99}
\item
A  {\em  left $ \Ae $-pairing\/}  is a  $ k $-bilinear  map  $\langle\ ,\ \rangle : U \times W \to A  $  such that for any  $  u \in U $,  $  w \in W $,  and  $  a \in A  $,  one has
\begin{equation*}
\begin{array}{rclrclrcl}
\langle u  , a \lact w \rangle   \!\!\!&=\!\!\!&
\langle u \ract a  , w \rangle, &
\langle u  , w \ract a \rangle   \!\!\!&=\!\!\!&
\langle a \blact u  , w \rangle, &
\langle u  , a \blact w \rangle  \!\!\!&=\!\!\!&
\langle u \bract a  , w \rangle, \\
\langle u  , w \bract a \rangle  \!\!\!&=\!\!\!&
\langle u  , w \rangle  a, &
\langle a \lact u  , w \rangle   \!\!\!&=\!\!\!&
 a \langle u  , w \rangle.
\end{array}
\end{equation*}
\item
A {\em  right $ \Ae $-pairing\/}  is a  $ k $-bilinear  map
$\langle\ ,\ \rangle : U \times W \to A  $  such that for any  $  u \in U $,  $  w \in W $,  and  $  a \in A  $,  one has
\begin{equation*}
\begin{array}{rclrclrcl}
\langle u  , w \ract a \rangle  \!\!\!&=\!\!\!&
\langle a \lact u  , w \rangle, &
\langle u  , a \lact w \rangle  \!\!\!&=\!\!\!&
\langle u \bract a  , w \rangle, &
\langle u  , w \bract a \rangle  \!\!\!&=\!\!\!&
\langle a \blact u  , w \rangle, \\
\langle u  , a \blact w \rangle   \!\!\!&=\!\!\!&
 a  \langle u  , w \rangle, &
\langle u \ract a  , w \rangle  \!\!\!&=\!\!\!&
\langle u  , w \rangle  a.
\end{array}
\end{equation*}
\end{enumerate}
\end{definition}

\begin{free text}{\bf Duals of bialgebroids.}
\label{regnetswirklich?}
 Let  $ U_{\scalastd} $  resp.\  $ U^{\scalast} $  be the left resp.\ right dual of a left bialgebroid.  If  $ \due U \lact {} $  is finitely generated projective, then  $ U_{\scalastd} $  is canonically endowed with a  {\em right} bialgebroid structure \cite{KadSzl:BAODTEAD} such that the evaluation pairing between  $ U $  and  $ U_{\scalastd} $  is a (nondegenerate)  {\em left\/}  pairing; similarly, if $U_\ract$ is finitely generated projective, then $ U^{\scalast} $  has a canonical  {\em right} bialgebroid structure  for which the natural pairing between  $ U $  and  $ U^\scalast $  is a  {\em right\/}  pairing.  If instead in either case the above finitely generated projective assumption is not satisfied, then both  $U^\scalast$  and  $ U_{\scalastd} $  are nevertheless  $ \Ae $-rings  endowed with a ``counit'' map, or augmentation.
\end{free text}

\subsection{Left and right Hopf algebroids}
\label{goeseveron}
\label{half-Hopf_algebroids}

 For any  left  bialgebroid  $ U  $,  define the  {\em Hopf-Galois maps}
\begin{equation*}
\begin{array}{rclrcl}
\ga_\ell : \due U \blact {} \otimes_{\Aopp} U_\ract &\to& U_\ract  \otimes_\ahha  \due U \lact,
& u \otimes_\Aopp v  &\mapsto&  u_{(1)} \otimes_\ahha u_{(2)}  v, \\
\ga_r : U_{\!\bract}  \otimes^\ahha \! \due U \lact {}  &\to& U_{\!\ract}  \otimes_\ahha  \due U \lact,
&  u \otimes^\ahha v  &\mapsto&  u_{(1)}  v \otimes_\ahha u_{(2)}.
\end{array}
\end{equation*}
%Similar maps can be defined for right bialgebroids but we refrain from writing them down explicitly.
With the help of these maps, we make the following definition due to Schauenburg \cite{Schau:DADOQGHA}:

\begin{definition}
\label{def Half Hopf bialgebroids}
A left bialgebroid $U$ is called a  {\em left Hopf algebroid} if $ \alpha_\ell  $ is a bijection.  Likewise, it is called a {\em right Hopf algebroid} if $\ga_r$ is so.
In either case, we adopt for all $u \in U$ the following (Sweedler-like)  notation
\begin{equation}
\label{latoconvalida}
u_+ \otimes_\Aopp u_-  :=  \alpha_\ell^{-1}(u \otimes_\ahha 1),  \qqquad
   u_{[+]} \otimes^\ahha u_{[-]}  :=  \alpha_r^{-1}(1 \otimes_\ahha u),
\end{equation}
and call both maps  $  u  \mapsto  u_+ \otimes_\Aopp u_-  $  and  $  u  \mapsto  u_{[+]} \otimes^\ahha u_{[-]}  $ {\em translation maps}.
\end{definition}

Analogous notions exist with respect to an underlying {\em right} bialgebroid structure, but we will not give the details here.

\begin{remark}
\label{left/right-Hopf-left-bialg_cocomm}
\noindent
\begin{enumerate}
\compactlist{99}
\item
In case $A=k$ is central in $U$, one can show that $\ga_\ell$ is invertible if and only if $U$ is a Hopf algebra, and the translation map reads
 $  u_+ \otimes u_-  :=  u_{(1)} \otimes S(u_{(2)})  $, where $S$ is the antipode of $U$.
On the other hand, $U$ is a Hopf algebra with invertible antipode if and only if both $\ga_\ell$ and $\ga_r$ are invertible,
and then $  u_{[+]} \otimes u_{[-]} := u_{(2)} \otimes S^{-1}(u_{(1)})  $.
% If the antipode is invertible, one also obtains the bijectivity of $\ga_r$, where
% $  u_{[+]} \otimes u_{[-]} := u_{(2)} \otimes S^{-1}(u_{(1)})  $.
%
%In either case, this shows that for bialgebras  the additional property of being a left or right Hopf algebroid is equivalent to being a Hopf algebra.
%\item
%By definition, a left bialgebroid  $ U $  is a left Hopf algebroid if and only if  $ U_\coop $  is
%a right Hopf algebroid.  In particular, for a cocommutative $U$ both notions of left and right Hopf algebroid coincide.
\item
The underlying left bialgebroid in a {\em full} Hopf algebroid with bijective antipode is both a left and right Hopf algebroid (but not necessarily vice versa); see \cite[Prop.\ 4.2]{BoeSzl:HAWBAAIAD} for the details of this construction.
\end{enumerate}
\end{remark}

   The following proposition collects some properties of the translation maps  \cite{Schau:DADOQGHA}:

\begin{prop}
Let $U$ be a left bialgebroid.
\begin{enumerate}
\compactlist{99}
\item If  $  U $  is a left Hopf algebroid, the following relations hold:
\begin{eqnarray}
\label{Sch1}
u_+ \otimes_\Aopp  u_- & \in
& U \times_\Aopp U,  \\
\label{Sch2}
u_{+(1)} \otimes_\ahha u_{+(2)} u_- &=& u \otimes_\ahha 1 \quad \in U_{\!\ract} \! \otimes_\ahha \! {}_\lact U,  \\
\label{Sch3}
u_{(1)+} \otimes_\Aopp u_{(1)-} u_{(2)}  &=& u \otimes_\Aopp  1 \quad \in  {}_\blact U \! \otimes_\Aopp \! U_\ract,  \\
\label{Sch4}
u_{+(1)} \otimes_\ahha u_{+(2)} \otimes_\Aopp  u_{-} &=& u_{(1)} \otimes_\ahha u_{(2)+} \otimes_\Aopp u_{(2)-},  \\
\label{Sch5}
u_+ \otimes_\Aopp  u_{-(1)} \otimes_\ahha u_{-(2)} &=&
u_{++} \otimes_\Aopp u_- \otimes_\ahha u_{+-},  \\
\label{Sch6}
(uv)_+ \otimes_\Aopp  (uv)_- &=& u_+v_+ \otimes_\Aopp v_-u_-,
\\
\label{Sch7}
u_+u_- &=& s^\ell (\varepsilon (u)),  \\
\label{Sch8}
\varepsilon(u_-) \blact u_+  &=& u,  \\
\label{Sch9}
(s^\ell (a) t^\ell (b))_+ \otimes_\Aopp  (s^\ell (a) t^\ell (b) )_-
&=& s^\ell (a) \otimes_\Aopp s^\ell (b),
\end{eqnarray}
where in  \rmref{Sch1}  we mean the Takeuchi-Sweedler product
\begin{equation*}
\label{petrarca}
   U \! \times_\Aopp \! U   :=
   \big\{ {\textstyle \sum_i} u_i \otimes v_i \in {}_\blact U  \otimes_\Aopp  U_{\!\ract} \mid {\textstyle \sum_i} u_i \ract a \otimes v_i = {\textstyle \sum_i} u_i \otimes a \blact v_i, \ \forall a \in A \big\}.
\end{equation*}
\item
Analogously, if $  U $  is a right Hopf algebroid, one has:
\begin{eqnarray}
\label{Tch1}
u_{[+]} \otimes^\ahha  u_{[-]} & \in
& U \times^{\scriptscriptstyle A} U,  \\
\label{Tch2}
u_{[+](1)} u_{[-]} \otimes_\ahha u_{[+](2)}  &=& 1 \otimes_\ahha u \quad \in U_{\!\ract} \! \otimes_\ahha \! {}_\lact U,  \\
\label{Tch3}
u_{(2)[-]}u_{(1)} \otimes^\ahha u_{(2)[+]}  &=& 1 \otimes^\ahha u \quad \in U_{\!\bract} \!
\otimes^\ahha \! \due U \lact {},  \\
\label{Tch4}
u_{[+](1)} \otimes^\ahha u_{[-]} \otimes_\ahha u_{[+](2)} &=& u_{(1)[+]} \otimes^\ahha
u_{(1)[-]} \otimes_\ahha  u_{(2)},  \\
\label{Tch5}
u_{[+][+]} \otimes^\ahha  u_{[+][-]} \otimes_\ahha u_{[-]} &=&
u_{[+]} \otimes^\ahha u_{[-](1)} \otimes_\ahha u_{[-](2)},  \\
\label{Tch6}
(uv)_{[+]} \otimes^\ahha (uv)_{[-]} &=& u_{[+]}v_{[+]}
\otimes^\ahha v_{[-]}u_{[-]},  \\
\label{Tch7}
u_{[+]}u_{[-]} &=& t^\ell (\varepsilon (u)),  \\
\label{Tch8}
u_{[+]} \bract \varepsilon(u_{[-]})  &=&  u,  \\
\label{Tch9}
(s^\ell (a) t^\ell (b))_{[+]} \otimes^\ahha (s^\ell (a) t^\ell (b) )_{[-]}
&=& t^\ell(b) \otimes^\ahha t^\ell(a),
\end{eqnarray}
where in  \rmref{Tch1}  we mean the Sweedler-Takeuchi product
\begin{equation*}  \label{petrarca2}
   U \times^{\scriptscriptstyle A} U   :=
   \big\{ {\textstyle \sum_i} u_i \otimes  v_i \in U_{\!\bract}  \otimes^\ahha \!  \due U \lact {} \mid {\textstyle \sum_i} a \lact u_i \otimes v_i = {\textstyle \sum_i} u_i \otimes v_i \bract a,  \ \forall a \in A  \big\}.
\end{equation*}
\end{enumerate}
\end{prop}

These two structures are not entirely independent:

\begin{lemma}
 The following mixed relations
 hold among left and right translation maps:
\begin{eqnarray}
\label{mampf1}
u_{+[+]} \otimes_\Aopp u_{-} \otimes^\ahha u_{+[-]} &=& u_{[+]+} \otimes_\Aopp u_{[+]-} \otimes^\ahha u_{[-]}, \\
\label{mampf2}
u_+ \otimes_\Aopp u_{-[+]} \otimes^\ahha u_{-[-]} &=& u_{(1)+} \otimes_\Aopp u_{(1)-} \otimes^\ahha u_{(2)}, \\
\label{mampf3}
u_{[+]} \otimes^\ahha u_{[-]+} \otimes_\Aopp u_{[-]-} &=& u_{(2)[+]} \otimes^\ahha u_{(2)[-]} \otimes_\Aopp u_{(1)},
\end{eqnarray}
where, for example, in the first equation \rmref{mampf1}
the second tensor product relates the first component with the third, and {\em mutatis mutandis} for the other identities.
\end{lemma}

\begin{proof}
 In order to prove  \rmref{mampf1},  we apply  $  \ga_\ell \otimes \id  $  to both sides (note that this operation is well-defined on the considered tensor products); for the right hand side we obtain, by definition,
 $$
(\ga_\ell \otimes \id)( u_{[+]+} \otimes_\Aopp u_{[+]-} \otimes^\ahha u_{[-]} )
=   (u_{[+]} \otimes_\ahha 1) \otimes^\ahha u_{[-]},
$$
and for the left hand side we have
\begin{equation*}
 \begin{split}     & (\ga_\ell \otimes \id)(u_{+[+]} \otimes_\Aopp u_{-} \otimes^\ahha u_{+[-]}) = (u_{+[+](1)}
 \otimes_\ahha u_{+[+](2)}u_-) \otimes^\ahha u_{+[-]} \\
            = \ & (u_{+(1)[+]}  \otimes_\ahha u_{+(2)}u_-) \otimes^\ahha u_{+(1)[-]}
            =   (u_{[+]} \otimes_\ahha 1) \otimes^\ahha u_{[-]},
 \end{split}
\end{equation*}
using  \rmref{Tch4}  and  \rmref{Sch2}.
% Applying  $  \alpha_\ell^{-1}  $ to the first and the third component of this identity gives the equality we are looking for.
Since  $ \ga_\ell $  is assumed to be an isomorphism, this proves  \rmref{mampf1}.

Let us also prove  \rmref{mampf2};  the remaining identity will be left to the reader.  To this end, apply  $  \id \otimes \ga_r  $  to both sides in  \rmref{mampf2}:  for the left hand side, we obtain
\begin{equation*}
\begin{split}
 (\id \otimes \ga_r)(u_{+} \otimes_\Aopp u_{-[+]} \otimes^\ahha u_{-[-]})   &=
u_{+} \otimes_\Aopp ( u_{-[+](1)}u_{-[-]} \otimes_\ahha u_{-[+](2)} )  \\
& =   u_+ \otimes_\Aopp (1 \otimes_\ahha u_-)
\end{split}
\end{equation*}
by  \rmref{Tch2},  and where in the second equation the first tensor product relates the first component with the third.  As for the right hand side, we compute:
\begin{equation*}
 \begin{split}
 & (\id \otimes \ga_r)(u_{(1)+} \otimes_\Aopp u_{(1)-} \otimes^\ahha u_{(2)}) = u_{(1)+} \otimes_\Aopp (u_{(1)-(1)} u_{(2)} \otimes_\ahha u_{(1)-(2)}) \\
            = \ & u_{(1)++} \otimes_\Aopp (u_{(1)-}u_{(2)} \otimes_\ahha u_{(1)+-})
=   u_{+} \otimes_\Aopp (1 \otimes_\ahha u_{-}),
 \end{split}
\end{equation*}
using  \rmref{Sch5}  and  \rmref{Sch3} in the last step as follows: Eq.~\rmref{Sch3}  yields
  $  u_{(1)+} \otimes_\Aopp u_{(1)-} u_{(2)} \otimes_\ahha 1 = u \otimes_\Aopp  1 \otimes_\ahha 1  $
and applying  $ \alpha_\ell^{-1} $  to the first and the third component gives the required equality.
\end{proof}

\section{Modules over left or right Hopf algebroids}
  \label{acquaprimavera}

   In this section we collect some general results about modules over left and right Hopf algebroids.  Some of them are known, while others  seem to have passed unnoticed so far (see Note \ref{osmanstoechter} below).

\subsection{Module structures on $\Hom$-spaces and tensor products}  \label{exotic_U-actions}
Similarly as for bialgebras,
the tensor product $M_\ract \otimes_\ahha \due {M'} \lact {}$ of two left $U$-modules with left $U$-module structure given by
\begin{equation}
\label{baratti&milano}
u(m \otimes_\ahha m') := u_{(1)} m \otimes_\ahha u_{(2)}m'
\end{equation}
equips the category $\umod$ for a left bialgebroid $U$ with a monoidal structure.
On the other hand, for $M \in \umod$ and $N \in \modu$, the $\Ae$-module $\Hom_\Aopp(M_\ract, N_\bract)$ is a right $U$-module via
$$
(fu)(m) := f(u_{(1)}m)u_{(2)}.
$$
The existence of a translation map if $U$ is, on top, a left or right Hopf algebroid makes it possible to endow $\Hom$-spaces and tensor products of $U$-modules with further natural $U$-module structures. The proof of the following  proposition  is straightforward.

\begin{proposition}
\label{structures}
Let $(U, A)$ be a left bialgebroid, $M, M' \in \umod$ and $N, N' \in \modu$ be left resp.\ right $U$-modules, denoting the respective actions by juxtaposition.
\begin{enumerate}
\compactlist{99}
\item
Let $(U,A)$ be additionally a left Hopf algebroid.
\begin{enumerate}
\compactlist{99}
\item
The $\Ae$-module $\Hom_{\Aopp}(M,M') $ carries a left  $ U $-module structure given by
\begin{equation}
\label{gianduiotto1}
(uf)(m)  :=  u_+ \big( f(u_-m) \big).
\end{equation}
%for $f \in \Hom_{\Aopp}(M,M'), \  m \in M$, and $u \in U$.
In particular,  $M^{\scalast} $  is endowed with a left  $ U $-module  structure.
\item
The $\Ae$-module $\Hom_\ahha(N,N')  $ carries a left $U$-module structure via
\begin{equation}
\label{lingotto1}
(u \rightslice f)(n)  :=  \big( f(nu_+) \big)u_-.
\end{equation}
%for $f \in \Hom_\ahha(N,N'), \  n \in N$, and $u \in U$.
\item
The $\Ae$-module $\due N \blact {} \otimes_\Aopp M_\ract$
carries a right  $ U $-module  structure via
\begin{equation}
\label{superga1}
(n \otimes_\Aopp m) \leftslice u  :=  nu_+ \otimes_\Aopp u_-m.
\end{equation}
%for $m \in M, \  n \in N$, and $u \in U$.
\end{enumerate}
\item
Let  $(U,A)$  be a right Hopf algebroid instead.
\begin{enumerate}
\compactlist{99}
\item
The $\Ae$-module  $  \Hom_\ahha(M,M')  $  carries a left  $ U $-module structure  given by
\begin{equation}
\label{gianduiotto2}
(uf)(m)  :=  u_{[+]}\big( f(u_{[-]}m) \big).
\end{equation}
%for $f \in \Hom_\ahha(M,M'), \  m \in M$, and  $u \in U$.
In particular,  $ M_{\scalastd} $  is naturally endowed with a left  $ U $-module  structure.
\item
The $\Ae$-module
$\Hom_{\Aopp}(N,N')  $  carries a left  $ U $-module structure  given by
\begin{equation}
\label{lingotto2}
(u \rightslice f)(n)  :=  \big( f(nu_{[+]}) \big)u_{[-]}.
\end{equation}
%for $f \in \Hom_\ahha(N,N'), \  n \in N$, and $u \in U$.
\item
The $\Ae$-module
$ N_{\bract} \otimes^\ahha \due M \lact {}  $
carries a right  $ U $-module structure  given by
\begin{equation}
\label{superga2}
(n \otimes^\ahha m) \leftslice u  :=  nu_{[+]} \otimes^\ahha u_{[-]}m.
\end{equation}
\end{enumerate}
\end{enumerate}
\end{proposition}

\begin{note}
\label{osmanstoechter}
These structures are well-known for  $ D $-modules  (that is, when  $  U = D_X  $, see \cite{Bor:ADM, Kas:DMAMC})  and were later extended to  $ V^\ell(L) $-modules  in \cite{Che:PDFKALS}, \cite{Chemla3}.  The results about tensor products can be found in  \cite{KowKra:DAPIACT}, whereas \rmref{gianduiotto1} serves in \cite[Thm.~3.5]{Schau:DADOQGHA} to characterise a possible (left) Hopf structure on a bialgebroid.
\end{note}

\subsection{Switching left and right modules: dualising modules}
\label{malgenauerhinsehen}

We investigate now conditions which imply an equivalence between the categories of left and of right  $ U $-modules  for a left bialgebroid  $ U $ which is simultaneously a left and right Hopf algebroid.  As in other frameworks, this is guaranteed by the existence of a suitable {\em dualising module}.  This is the content of the next result, which generalises the well-known equivalence of categories between left and right  $ \cD $-modules  (due to Borel \cite{Bor:ADM} and Kashiwara \cite{Kas:DMAMC}). It also generalises the equivalence between left and right modules over a Lie-Rinehart algebra, {\em cf.}~\cite{Che:PDFKALS}.

\begin{proposition}
\label{left and right modules}
 Let  $(U,A)$  be simultaneously a left  and  right Hopf algebroid.
Assume that there exists a right  $ U $-module  $ P $, where $ P_\bract $  is finitely generated projective over $\Aop$, such that
\begin{enumerate}
\compactlist{99}
\item
the left $U$-module morphism
 $$
A \to \Hom_{\Aopp}(P,P),  \quad  a  \mapsto  \{ p \mapsto a \blact p \}
$$
 is an isomorphism of $k$-modules;
\item
the evaluation map
\begin{equation}
\label{lacartachenontagliaglialberi}
\due P \blact {} \otimes_\Aopp \Hom_{\Aopp}(P  , N)_{\ract} \to N, \quad p  \otimes_\Aopp  \phi  \mapsto  \phi(p)
\end{equation}
is an isomorphism for any $ N \in \modu $.
\end{enumerate}
   Then
$$
\umod \to \modu, \quad M \mapsto \due P \blact {} \otimes_\Aopp M_\ract
$$
is an equivalence of categories with quasi inverse given by $  N'  \mapsto  \Hom_{\Aopp}(P , N')  $.
\end{proposition}

\begin{proof}
For $M \in \umod$ and $N, N' \in \modu $, one checks with \rmref{mampf3} that
the map
  $$
M_\ract \otimes_\ahha {}_\lact \Hom_{\Aopp}(N, N') \to \Hom_{\Aopp}(N,  \due {N'} \blact {} \otimes_\Aopp M_\ract), \
m \, \otimes_\ahha \chi \mapsto \{n \mapsto \chi(n) \otimes_\Aopp m \}
$$
is a morphism of left  $U $-modules, where the left $U$-module structure on the left hand side is given by \rmref{baratti&milano} combined with \rmref{lingotto2}, and on the right hand side by \rmref{lingotto2} combined with \rmref{superga1}. It is even an isomorphism if  $N_\bract $  is finitely generated projective over $A$.
On the other hand, using  \rmref{mampf2}  and  \rmref{Sch7},  one easily sees that
the evaluation \rmref{lacartachenontagliaglialberi}
is a morphism of right  $ U $-modules;  it is then an isomorphism by hypothesis, which finishes the proof.
\end{proof}

% \begin{remark}
% Of course, a similar statement holds when interchanging the r\^oles of source and target, {\em i.e.}, if $\due P \blact {}$
% is projective of finite type and if the $U$-module morphism $A \to \Hom_\ahha(P,P), \ a \mapsto \{ p \mapsto p \bract a \}$
% is an isomorphism, but we will not give the details here.
% \end{remark}

\begin{remark}
A right $U$-module $P$ with the properties as in the above proposition appeared in various contexts in the literature: we shall call it a  {\em dualising module}.  We refer to \cite{Che:PDFKALS, KowKra:DAPIACT, Hue:DFLRAATMC} for applications and details, and in particular to the situation in \S\ref{heisz}.
\end{remark}

\section{Comodule equivalences and induced maps between duals}

\label{fallouts}

The aim of this section is to construct a map between the left and right dual of a left Hopf algebroid, which in some sense replaces the missing antipode on either of the duals. This can be essentially done in two ways, either by a quite straightforward generalisation of the antipode construction on the dual of a cocommutative left Hopf algebroid as in \cite{KowPos:TCTOHA}, or by considering Ph\`ung's comodule equivalence in \cite{Phu:TKDFHA} as a starting point, as suggested by the referee of the present paper. To pursue the latter approach, we will review and slightly extend the results in {\em op.~cit.}

\subsection{A categorical equivalence for comodules}

The following theorem, originally due to \cite{Phu:TKDFHA}, shows that under the given conditions every right $U$-comodule can be transformed into a left one (resp.\ vice versa in the second case). We repeat it here for future use and also slightly extend it by saying that the two given functors are quasi-inverse to each other and that they are (strict) monoidal:

\begin{theorem}
\label{U-comod <-> comod-U}
Let $(U,A)$ be a left bialgebroid.

\begin{enumerate}
\compactlist{99}
\item
Let $(U,A)$ be additionally a left Hopf algebroid such that $U_{\ract}$ is projective.
%, where $U_\ract$ is finitely generated projective over $A$.
Then there exists a (strict) monoidal functor  $F:  \comodu \to \ucomod  $;  namely, if $M$ is a right $U$-comodule with coaction
%$\rho: M \to M \otimes_\ahha \due U \lact {}, \
$m \mapsto m_{(0)} \otimes_\ahha m_{(1)}  $,  then
\begin{equation}
\label{borromini}
\lambda_\emme: M \to U_\ract \otimes_\ahha M  , \quad m \mapsto  m_{(1)-} \otimes_\ahha  m_{(0)} \epsilon(m_{(1)+})  ,
\end{equation}
defines a left comodule structure on $M$ over $U$.
\item
Let  $(U,A)$  be a right Hopf algebroid such that $_{\lact}U$ is projective.
%, where $\due U \lact {}$ is finitely generated projective over $A$.
 Then there exists a (strict) monoidal functor  $ G: \ucomod \to \comodu  $;  namely, if $N$ is a left $U$-comodule with coaction
%$\gl: N \to U_\ract \otimes_\ahha N, \
$n \mapsto n_{(-1)} \otimes_\ahha n_{(0)}  $,  then
\begin{equation}
\label{bernini}
\rho_\enne: N \to N \otimes_\ahha \due U \lact {}  , \quad n \mapsto \epsilon(n_{(-1)[+]}) n_{(0)} \otimes_\ahha  n_{(-1)[-]}  ,
\end{equation}
defines a right comodule structure on $N$ over $U$.
\item
 If $U$ is both a left and right Hopf algebroid and if both $U_{\ract}$ and $_{\lact}U$ are $A$-projective,
%and both $U_\ract$ and  $\due U \lact {}$ are finitely generated projective over $A$
then the functors mentioned in {\it (i)} and {\it (ii)} are quasi-inverse to each other and we have an equivalence
$$
\ucomod \simeq \comodu
$$
of monoidal categories.
\end{enumerate}
\end{theorem}

\begin{proof}
% [Proof of Theorem \ref{U-comod <-> comod-U}]
% \footnote{basically, I inserted what Niels wrote but I changed it a little bit}
% (i)
%
 Let us first prove that (\ref{borromini}) is well defined.
For any right $U$-comodule  $ M $  with coaction  $\rho: M \to M \otimes_\ahha U$, there is a well-defined map $\id_\emme \otimes_\ahha \epsilon: M \otimes_\ahha U \to M$. Its restriction to the Takeuchi product $M \times_\ahha U$ is a left $A$-module map as shows the following equation:  for any $\sum_i m_i \otimes u_i \in M \times_\ahha U$ and any $a \in A$, one has
$$
\sum_i m_i \eps(a \blact u_i) = \sum_i m_i \eps(u_i \bract a)  = \sum_i a m_i \eps(u_i).
$$
Thus, there is a well-defined map
$$
% \widehat{\id_\emme \otimes_\ahha \eps}:
\id_\emme \times_\ahha \eps : \,
M \times_\ahha U \to M, \quad
\sum_i m_i \otimes u_i \mapsto \sum_i m_i \eps(u_i),
$$
% and in particular, the map
and hence, in particular, the map 
%%
% $$
% (\widehat{\id_\emme \otimes_\ahha \eps}) \otimes_\Aopp \id_\uhhu:
% (M \times_\ahha U) \times_\Aopp U  \to M \times_\Aopp U
%  \eqno{(\heartsuit)}
% $$
\begin{equation}  
\label{def-phi}
 \phi := (\id_\emme \times_\ahha \eps) \otimes_\Aopp \id_\uhhu :  \,
(M \times_\ahha U) \times_\Aopp U  \to M \times_\Aopp U
\end{equation}
is well-defined, too.

On the other hand, any right coaction corestricts to a map $M \to M \times_\ahha U$;
similarly, the translation map
$\beta^{-1}(u \otimes_\ahha 1) = u_+ \otimes_\Aopp u_-$ of  $ U $  corestricts to a map $U \to U \times_\Aopp U $. Combining these two maps gives a map
%%
% $$
% M \to M \times_\ahha (U \times_\Aopp U)   \eqno{(\spadesuit)}
% $$
%%
\begin{equation}  
\label{def-psi}
\psi:  M \to M \times_\ahha (U \times_\Aopp U), 
\end{equation}
and it is clear that if we could combine $\phi$ in \rmref{def-phi}
with $\psi$ in \rmref{def-psi}
followed by a tensor flip, this would
yield the map \rmref{borromini}.

Now the problem is that usually $(M \times_\ahha U) \times_\Aopp U $ and $M \times_\ahha (U \times_\Aopp U)$ are different, hence the two maps might not be composable. Introducing as in \cite[Def.~1.4]{Tak:GOAOAA} the triple Takeuchi product 
\begin{equation*}
\begin{split}
   M \times_\ahha U \times_\Aopp U   :=  
\{\textstyle\sum_i & m_i  \otimes u_i\otimes v_i  \in M \otimes_\ahha U \otimes_\Aopp U \mid \\
& \textstyle \sum_i  am_i \otimes u_i \ract b \otimes v_i 
=  \sum_i m_i \otimes u_i \bract a \otimes b \blact v_i, \ \forall \ a,b \in A \}.
\end{split}
\end{equation*}
It can be seen that $ \psi $ 
actually maps to  $M \times_\ahha U \times_\Aopp U $ 
but
it is a priori not clear whether $ \phi $  can be directly defined on $M \times_\ahha U \times_\Aopp U $ so as to make the two maps composable.

However, in any case there are always maps
  $$  
\gamma: M \times_\ahha (U \times_\Aopp U) \to M \times_\ahha U \times_\Aopp U,  \quad m \otimes_\ahha u \otimes_\Aopp v \mapsto m \otimes_\ahha u \otimes_\Aopp v
$$
and
  $$
\ga: (M \times_\ahha U) \times_\Aopp U \to M \times_\ahha U \times_\Aopp U,  \quad  m \otimes_\ahha u \otimes_\Aopp v \mapsto m \otimes_\ahha u \otimes_\Aopp v.
$$

If now $U_\ract$ is projective, $\ga$ is an isomorphism \cite[Prop.~1.7]{Tak:GOAOAA}; then the composition $\tau \circ \phi
\circ \alpha^{-1} \circ \gamma \circ \psi$  
of well-defined maps (where  $ \tau $  is the tensor flip) yields a well-defined map again, and  on an element $m \in M$ it is an easy check that this gives the map $\gl_\emme$ in \rmref{borromini}.

%On a single  $ m \in M $  this reads  $ m \mapsto  m_{(1)-} \otimes_\ahha m_{(0)}\eps(m_{(1+)})$,  that is \rmref{borromini}.

That the so-defined map $\gl_\emme$ is $\Ae$-linear follows from the $\Ae$-linearity of the right coaction along with \rmref{Sch9}. That $\lambda_\emme$ indeed defines a left $U$-coaction is an easy check  using \rmref{Sch5} and \rmref{Sch4}, the counitality of the bialgebroid $U$, and the coassociativity with the $\Ae$-linearity of the right $U$-coaction on $M$ again: we have for $m \in M$
\begin{equation*}
\begin{split}
(\gD_\ell &\otimes \id)\lambda_\emme(m)
=  m_{(1)-(1)} \otimes_\ahha  m_{(1)-(2)} \otimes_\ahha  m_{(0)} \epsilon(m_{(1)+}) \\
&=  m_{(1)-} \otimes_\ahha  m_{(1)+-} \otimes_\ahha  m_{(0)} \epsilon(m_{(1)++}) \\
&=  m_{(1)-} \otimes_\ahha  \big( t^\ell \epsilon(m_{(1)+(2)}) m_{(1)+(1)} \big)_- \otimes_\ahha  m_{(0)} \epsilon\big(\big(t^\ell \epsilon(m_{(1)+(2)}) m_{(1)+(1)}
 \big)_+\big) \\
&=  m_{(2)-} \otimes_\ahha  \big( t^\ell \epsilon(m_{(2)+}) m_{(1)} \big)_- \otimes_\ahha  m_{(0)} \epsilon\big(\big( t^\ell \epsilon(m_{(2)+}) m_{(1)} \big)_+\big)
\\
&= (\id \otimes  \lambda_\emme)\lambda_\emme(m).
\end{split}
\end{equation*}
The counitality of $\lambda_\emme$ follows from \rmref{Takeuchicoaction} along with the second equation in \rmref{castelnuovo}.

   As for the claim that the so-given functor $F: \comodu \to \ucomod$ is (strict) monoidal, observe first that for any two $M, M'$ in the monoidal category $\comodu$, their tensor product  $ M \otimes_\ahha M' $  is a right $U$-comodule by means of the codiagonal coaction $m \otimes_\ahha m' \mapsto (m_{(0)} \otimes_\ahha m'_{(0)}) \otimes_\ahha  m'_{(1)}  m_{(1)}$,
that is, with a flip in the factors in $U$. On the other hand, the tensor product of two $N, N'$ in the monoidal category $\ucomod$ becomes a left $U$-comodule again
via
$n \otimes_\ahha n' \mapsto  n_{(-1)}  n'_{(-1)} \otimes_\ahha (n_{(0)} \otimes_\ahha n'_{(0)})$.
By the bialgebroid properties, \rmref{Sch6}, and \rmref{Takeuchicoaction}
it is then simple to see that
\begin{equation*}
 \begin{split}
( m'_{(1)}  m_{(1)})_- &\otimes_\ahha ( m_{(0)} \otimes_\ahha  m'_{(0)}) \epsilon\big(( m'_{(1)}  m_{(1)})_+\big)
\\
&=
 m_{(1)-}  m'_{(1)-} \otimes_\ahha \big( m_{(0)} \otimes_\ahha  m'_{(0)} \epsilon\big(m'_{(1)+} s^\ell(\epsilon(m_{(1)+}))\big)  \big)
\\
&=
 m_{(1)-}  m'_{(1)-} \otimes_\ahha \big( m_{(0)}\epsilon(m_{(1)+}) \otimes_\ahha  m'_{(0)} \epsilon(m'_{(1)+}) \big),
\end{split}
\end{equation*}
that is, $F(M \otimes_\ahha M') = F(M) \otimes_\ahha F(M')$. Also, the unit object in both $\comodu$ and $\ucomod$ is given by $A$ with coaction $a \mapsto t^\ell(a)$ resp.\ $a \mapsto s^\ell(a)$, and $F(A) = A$ now follows from \rmref{Sch9}. Moreover, note that $F$ does not affect the underlying $\Ae$-module structures of the comodules in question, and hence its (strict) monoidality follows.

The proof of {\em (ii)\/} is similar, and the last claim follows by the preceding two combined with
%Theorem \ref{criptabalbi} (or
a direct computation: applying $GF$ to a right comodule $M \in \comodu$, the resulting right coaction on $M$ reads
$$
M \to M \otimes_\ahha \due U \lact {}, \quad m \mapsto
\epsilon(m_{(1)-[+]}) m_{(0)} \epsilon(m_{(1)+}) \otimes_\ahha m_{(1)-[-]}.
$$
By using \rmref{mampf2}, the coassociativity and counitality of the original right coaction on $M$, \rmref{Takeuchicoaction}, \rmref{castelnuovo}, and \rmref{Sch8} one obtains
\begin{equation*}
\begin{split}
\epsilon(m_{(1)-[+]}) m_{(0)} \epsilon(m_{(1)+}) \otimes_\ahha m_{(1)-[-]} &= \epsilon(m_{(1)-}) m_{(0)} \epsilon(m_{(1)+}) \otimes_\ahha m_{2} \\
&= m_{(0)} \epsilon( \epsilon(m_{(1)-}) \blact m_{(1)+}) \otimes_\ahha m_{(2)} \\
&= m_{(0)} \epsilon(m_{(1)})  \otimes_\ahha m_{(2)} =  m_{(0)} \otimes_\ahha m_{(1)},
\end{split}
\end{equation*}
that is, the right coaction on $M$ we started with. An analogous consideration holds for $FG$ using \rmref{mampf3}, \rmref{Tch8}, and the Takeuchi property that holds for left $U$-comodules analogous to \rmref{Takeuchicoaction}.
\end{proof}

\begin{remark}
Note that the equivalence in  Theorem \ref{U-comod <-> comod-U}  does {\em not\/}  boil down to the usual equivalence of left and right comodules via the antipode (as there is no antipode for left or right Hopf algebroids, not even if the bialgebroid is simultaneously both). Even if we dealt with a full Hopf algebroid, this is still a different kind of equivalence (compared to
 the construction in \cite[Remark 4.6]{Boe:HA}), as follows from the considerations in \S\ref{napoleon} and \S\ref{enoteca} below.
For example, if the left Hopf algebroid $U$ is considered a right comodule over itself via the coproduct, the left $U$-coaction on $U$ from \rmref{borromini} is given by
$$
U \to U_\ract \otimes_\ahha \due U \blact {}, \quad u \mapsto u_- \otimes_\ahha u_+,
$$
that is, the ``flipped'' translation map.
On the other hand,
for Hopf {\em algebras} the construction in Theorem \ref{U-comod <-> comod-U} is exactly the equivalence induced by the antipode.
\end{remark}

\subsection{Constructing maps between the duals}
\label{wiley}

We now want to construct a map between the right and the left dual of a left Hopf algebroid. To this end, we first need to recall from
 \cite[Theorem 3.1.11]{Kow:HAATCT}
the following bialgebroid generalisation of the classical bialgebra module-comodule correspondence, which, however, in its first part comes somewhat unexpected at first sight:

\begin{prop}
\label{KatCM1}
Let $(U,A)$ be a left bialgebroid.
\begin{enumerate}
\compactlist{99}
\item
There exists a functor  $  \comodu \to \mathbf{Mod}\mbox{-}{U_{\scalastd}}  $;  namely, if $M$ is a right $U$-comodule with coaction
%$\rho: M \to M \otimes_\ahha \due U \lact {}, \
$m \mapsto m_{(0)} \otimes_\ahha m_{(1)}  $, then
\begin{equation}
\label{sale}
M \otimes_k U_\scalast \to M  , \quad m \otimes_k \psi \mapsto m_{(0)}\psi(m_{(1)})  ,
\end{equation}
defines a right module structure over the $\Ae$-ring $U_\scalast$.
If $\due U \lact {}$ is finitely generated $A$-projective
%%%
(so that  $ U_\scalast $  is a right bialgebroid),
%%%
 this
 functor is monoidal and has a quasi-inverse $\mathbf{Mod}\mbox{-}{U_{\scalastd}} \to \comodu$ such that there is an
 equivalence
 $
 \comodu \simeq \mathbf{Mod}\mbox{-}{U_{\scalastd}}
 $
 of categories.
\item
Likewise, there exists a functor  $  \ucomod \to \mathbf{Mod}\mbox{-}{U^{\scalast}}  $;  namely, if $N$ is a left $U$-comodule with coaction
%$\gl: N \to U_\ract \otimes_\ahha N, \
$n \mapsto n_{(-1)} \otimes_\ahha n_{(0)}  $, then
\begin{equation}
\label{pepe}
N \otimes_k U^\scalast \to N  , \quad n \otimes_k \phi \mapsto \phi(n_{(-1)})n_{(0)}  ,
\end{equation}
defines a right module structure over the $\Ae$-ring $U^\scalast$.
If $\due U {}\ract$ is finitely generated
 $A$-projective
%%%
(so that  $ U^\scalast $  is a right bialgebroid),
%%%
 this functor is monoidal and has a quasi-inverse $\mathbf{Mod}\mbox{-}{U^{\scalast}} \to \ucomod$ such that there
 is an equivalence
 $
 \ucomod \simeq \mathbf{Mod}\mbox{-}{U^{\scalast}}
 $
 of categories.
\end{enumerate}
\end{prop}

The case $(ii)$ of the above Proposition \ref{KatCM1} can also be found in \cite[\S5]{Schau:DADOQGHA}.
An explicit proof and a description of all involved functors is given in \cite[\S3.1]{Kow:HAATCT}, along with the respective structure maps of the right bialgebroids $(U_{\scalastd}, A, s_{\scalastd}^r, t_{\scalastd}^r, \gD_{\scalastd}^r, \pl_{\scalastd})$ and $(U^{\scalast}, A, s_r^\scalast, t_r^\scalast, \gD^{\scalast}_r, \pl^{\scalast})$, in case the respective mentioned finiteness assumptions are met.
Observe that when $(U,A)$ is both a left and a right Hopf algebroid and both $U_\ract$ as well as $\due U \lact {}$ are finitely generated projective over $A$, then \rmref{eastpak} here
below
 is a commutative diagram of monoidal equivalences.

We shall also need an explicit expression of the induced coaction on $M \in \mathbf{Mod}\mbox{-}{U_{\scalastd}}$ in case $\due U \lact {}$ is finitely generated projective as in {\em (i)}: let $m \otimes_k \psi \mapsto m\psi$ denote the right $U_{\scalastd}$-action on $M$ and
$\{e_i\}_{1 \leq i \leq n} \in U, \ \{ e^i\}_{1 \leq i \leq n} \in U_{\scalastd}$ a dual basis (see, for example, \cite[p.~202]{AndFul:RACOM} for the notion of dual basis of a finitely generated projective module).
Then the resulting right $U$-coaction on $M$ can be expressed as
\begin{equation}
\label{ferra}
 m \mapsto \sum_i m e^i \otimes_\ahha e_i,
\end{equation}
see \cite[Eq.~(3.1.23)]{Kow:HAATCT}.
Consider now the diagram
\begin{equation}
\label{eastpak}
\begin{gathered}
\xymatrix{ \comodu \ar@{->}[r]  \ar@{->}[d] &  \mathbf{Mod}\mbox{-}U_{\scalast} \ar@{..>}[d] \\ \ucomod \ar@{->}[r] %\ar@<4pt>[u]
& \mathbf{Mod}\mbox{-}U^{\scalast}
}
\end{gathered}
\end{equation}
of categories, where the left vertical arrow is that from Theorem \ref{U-comod <-> comod-U} {\em (i)}.
Under the finiteness assumption for $\due U \lact {}$, the upper horizontal arrow is invertible.
One therefore 
obtains a functor that corresponds to the dotted arrow if $ U_\ract{}$ is $A$-projective and
$\due U \lact {}$ is finitely generated $A$-projective.
% that is, every right $U_{\scalast}$-module yields a right $U^{\scalast}$-module.
Explicitly, by using \rmref{ferra}, \rmref{borromini}, and \rmref{pepe} one obtains on a right $U_{\scalast}$-module $M$
with $U_{\scalast}$-action $m \otimes_k \psi \mapsto m\psi$ the following right $U^{\scalast}$-action:
\begin{equation}
\label{dieda}
M \otimes_k U^{\scalast} \to M, \quad m \mapsto m \Yleft \phi := \phi({e_i}_-) m e^i \epsilon({e_i}_+) = me^i \epsilon\big({e_i}_+ s^\ell(\phi({e_i}_-))\big),
\end{equation}
where the second expression follows by taking the Takeuchi property \rmref{Takeuchicoaction} of the right coaction \rmref{ferra} into account, along with \rmref{Sch9}.

Consider now the case $M = U_{\scalast}$ as right module over itself by right multiplication;
then as in \rmref{dieda} it also carries a right  $ U^\scalast $-action,
which is equivariant with respect to the regular left  $ U_{\scalast} $-action, that is
\begin{equation}
\label{equivariance}
 (\psi' \psi'') \Yleft \phi = \psi' ( \psi'' \Yleft \phi).
\end{equation}
In particular, this implies $  \psi \Yleft \phi = \psi ( 1_{U_{\scalast}} \Yleft \phi )$, which leads us to consider
\begin{equation}
\label{1st-def_Sstarup}
  S^\scalast\phi := 1_{U_{\scalast}} \Yleft \phi = \epsilon \Yleft \phi.
\end{equation}
 With \rmref{dieda}, we see that
%\begin{equation*}
%\begin{split}
$
S^\scalast\phi = \epsilon \Yleft \phi =  e^i s^r_\scalast\big(\epsilon\big({e_i}_+ s^\ell(\phi({e_i}_-))\big)\big).
$
%\end{split}
%\end{equation*}
Hence, for any $u \in U$,
\begin{equation}
\label{sososo}
\begin{split}
S^\scalast\phi(u) &= \langle \epsilon \Yleft \phi, u\rangle = \big\langle e^i s^r_\scalast\big(\epsilon\big({e_i}_+ s^\ell(\langle \phi, {e_i}_- \rangle)\big)\big), u \big\rangle \\
& = \langle e^i, u \rangle \langle \epsilon, {e_i}_+ s^\ell(\langle \phi, {e_i}_-
\rangle) \rangle =\langle \epsilon, s^\ell( \langle e^i, u \rangle ) {e_i}_+ t^\ell(\langle \phi, {e_i}_-
\rangle) \rangle,
\end{split}
\end{equation}
where we used \cite[Eq.~(3.1.3)]{Kow:HAATCT} in the third step and \rmref{castelnuovo} in the fourth.
Inserting now into \rmref{sososo} the identity
$$
u_+ \otimes_\Aopp u_- = s^\ell( \langle e^i, u \rangle ) {e_i}_+ \otimes_\Aopp {e_i}_-,
$$
which is seen by applying the bijective Hopf-Galois map $\alpha_\ell$ from \rmref{latoconvalida} to both sides (as we assumed $U$ to be a left Hopf algebroid), one further obtains
\begin{equation}
\label{S*-prov}
S^\scalast\phi(u) =\langle \epsilon, s^\ell( \langle e^i, u \rangle ) {e_i}_+ t^\ell(\langle \phi, {e_i}_- \rangle) = \epsilon\big(u_+ t^\ell(\phi(u_-))\big).
\end{equation}

As will be discussed at length in the next section, this yields a map $S^\scalast: U^{\scalast} \to U_{\scalast}$ (as is seen using \rmref{Sch9} and \rmref{castelnuovo}) of $\Ae$-rings that even makes sense
% \footnote{here, I changed}
without any projectiveness or finiteness assumptions.

By means of \rmref{sale} and \rmref{S*-prov}, the action \rmref{dieda} can then be written as
\begin{equation}
\label{diqua}
m \Yleft \phi := m S^\scalast(\phi),
\end{equation}
which, without assuming any finiteness conditions on $U$, still leads to a functor $\mathbf{Mod}\mbox{-}U_{\scalast} \to  \mathbf{Mod}\mbox{-}U^{\scalast}$ between the categories of modules over $\Ae$-rings.

If instead $ U $  is a {\em right\/}  Hopf algebroid,
% \footnote{here, I changed a little}
where $U_\ract$ is
finitely generated $A$-projective and
$\due U \lact {}$ is $A$-projective,
one obtains by analogous steps a map $S_\scalast : U_{\scalast} \to U^{\scalast}  $ given by
  $$
S_\scalast\psi(u) = \epsilon\big(u_{[+]} s^\ell(\psi(u_{[-]}))\big)
$$
for any $u \in U$, to which analogous comments apply as above.

We will discuss the properties of these maps in detail in the subsequent \S\ref{linking-structures}

\section{Linking structure for the duals of left Hopf algebroids}
 \label{linking-structures}

In this section --- the core of the present work ---, we find that the map $S^\scalast$ constructed in the previous section is linking the right dual to the left dual of a left Hopf algebroid, which is apparently as close as one can get to an explicit formula of an antipode kind-of structure on the dual. Note, however, that even in the case of a full Hopf algebroid this map is not simply the transpose of the antipode, as discussed in \S\ref{napoleon}.
In some sense, this special map amounts to sort of a generalisation of (the antipode in) a full Hopf algebroid as explained in Remark \ref{schonschoen}.

As mentioned before, the definition of the map  $ S^\scalast $  (and  $ S_\scalast $)  actually makes sense even without any
% \footnote{here I changed}
 finiteness or projectiveness  assumptions.  Indeed, one can trace their first appearance already in  \cite{KowPos:TCTOHA} in the r\^ole of the antipode in the example of the bialgebroid of jet spaces.

In what follows, we will prove the fact that $S^\scalast$ and $S_\scalast$ are morphisms of $\Ae$-rings in a direct way, whereas the fact that under suitable finiteness assumptions they are bialgebroid morphisms is shown by using the comodule equivalence discussed
 in the previous section (note, however, that even the latter can be achieved by direct computation).

In particular, since the finiteness assumptions are not needed for all properties stated below, we will be able to apply  $ S^\scalast $  and  $ S_\scalast $  in greater generality to the examples in \S\ref{lautesGequatsche}.

\subsection{Morphisms between left and right duals}
\label{linking}

Let  $(U,A) $  be a  left bialgebroid.  If it is additionally a left Hopf algebroid,
its right dual $ U^{\scalast} $ (see \S\ref{mayitbe}) carries a  left $U$-module structure as in  \rmref{gianduiotto1};
(re-)define
\begin{equation}
\label{sstarup}
  S^{\scalast}(\phi)(u)   :=   (u\phi)(1_\uhhu)   =   \epsilon_\uhhu \big( u_+  t^\ell(\phi(u_-)) \big),   \qquad  \forall  \phi \in U^{\scalast}  ,  \ u \in U.
\end{equation}
Likewise, if the left bialgebroid $(U,A) $  is a right  Hopf algebroid instead, its left dual  $ U_{\scalastd} $ (see \S\ref{mayitbe} again) carries a left $U$-module structure as in \rmref{gianduiotto2}, with the help of which one (re-)defines
\begin{equation}
\label{sstardown}
  S_{\scalastd}(\psi)(u)   :=   (u\psi)(1_\uhhu)   =   \epsilon_\uhhu \big( u_{[+]}  s^\ell( \psi(u_{[-]})) \big),   \qquad  \forall  \psi \in U_{\scalastd}  ,  \ u \in U.
\end{equation}

The following result presents the key properties of the maps  $ S^{\scalast} $  and  $ S_{\scalastd}  $:

\begin{theorem}
\label{S*-morphism}
Let $(U,A)$ be a left bialgebroid.
\begin{enumerate}
\compactlist{99}
\item
  If  $(U,A) $  is moreover a left Hopf algebroid, \rmref{sstarup}  defines a morphism
$  S^{\scalast} : U^{\scalast} \to U_{\scalastd}  $
% such that  $  (S^{\scalast}, \id_\ahha): (U^{\scalast}, A) \to (U_{\scalastd}, A)  $
% is a morphism
of  $ A^e $-rings  with augmentation (the ``counit''); if in addition both  $ \due U \lact {} $  and  $ U_\ract $  are finitely generated projective as  $ A $-modules,  then  $ (S^{\scalast}, \id_\ahha) $  is a morphism of right bialgebroids.  In any case,  $ S^{\scalast} $  is also a morphism of left  $U$-modules  for the action  \rmref{lingotto1}  on  $ U^{\scalast} $  and the left action  on $ U_{\scalast} $ given by right multiplication in $U$.
\item
  If  $(U,A) $  is a right Hopf algebroid instead, \rmref{sstardown}  defines a morphism
$  S_{\scalastd} : U_{\scalastd} \to U^{\scalast}  $
% such that  $  (S_{\scalastd}, \id_\ahha): (U_{\scalastd}, A) \to  (U^{\scalast},A)  $  is a morphism
of  $ A^e $-rings  with augmentation (the ``counit''); if in addition both  $ \due U \lact {} $  and  $ U_\ract $  are finitely generated projective as  $ A $-modules,  then  $ (S_{\scalast}, \id_\ahha) $  is a morphism of right bialgebroids.  In any case,  $ S_{\scalastd} $  is also a morphism of left  $ U $-modules  for the action  \rmref{lingotto2} on  $ U_{\scalastd} $  and the left action on  $ U^{\scalast} $ given by right multiplication in $U$.
\end{enumerate}
\end{theorem}

\begin{proof}
 We only prove part {\it (i)} as {\it (ii)} follows  {\em mutatis mutandis}.  For the explicit computations, we will again use the notation and description of the structure maps of the two right bialgebroids
$(U_{\scalastd}, A, s_{\scalastd}^r, t_{\scalastd}^r, \gD_{\scalastd}^r, \pl_{\scalastd})$ and $(U^{\scalast}, A, s_r^\scalast, t_r^\scalast, \gD^{\scalast}_r, \pl^{\scalast})$
  --- where the coproduct  $ \gD_{\scalastd}^r $  or  $ \gD^{\scalast}_r $ only make sense if $U_\ract$ resp.\ $\due U \lact {}$ is finitely generated $A$-projective --- as given in detail in \cite[\S3.1]{Kow:HAATCT},  together with the respective properties of left and right pairings $\langle .,. \rangle$ as in Definition \ref{left and right 1}.
Direct verification shows that $S^{\scalast}$ takes values in $U_{\scalastd}$. Besides,
for $S^{\scalast}$ to be a bialgebroid morphism, we need to show the following properties:
  $$  \displaylines{
   \qquad
\text{\it (a)}  \hfill   S^{\scalast} s^{\scalast}_r   =   s_{\scalastd}^r  ,
\quad  S^{\scalast} t_r^{\scalast}   =   t_{\scalastd}^r, \quad
\pl_{\scalastd}  S^{\scalast} = \pl^{\scalast},
\hfill  \cr
   \qquad
\text{\it (b)}  \hfill   S^{\scalast}\big(\phi  \phi'\big)   =   S^{\scalast} (\phi)  S^{\scalast} (\phi')   \hfill  \cr
   \qquad
\text{\it (c)}  \hfill   \gD^r_{\scalastd}  S^{\scalast}   =   (S^{\scalast} \otimes S^{\scalast})  \gD^{\scalast}_r  ,  \hfill  }
$$
(where, as said before, {\it (c)\/} only makes sense if $U_\ract$ and $\due U \lact {}$ are finitely generated $A$-projective).

As for  {\em (a)},  we find for $u \in U$, $a \in A$ by direct computation using
\rmref{Sch8} and \rmref{Sch9}:
\begin{equation*}
\begin{split}
S^{\scalast}\big(s_r^{\scalast}(a)\big)(u)   &=   \epsilon\big( u_+  t^\ell\big(s_r^{\scalast}(a)(u_-) \big)\big)
= \epsilon \big( u_+  t^\ell\big(\epsilon(u_-  s^\ell(a)) \big)\big)
% \\
%&=   \epsilon\big( u_+  u_-  s^\ell(a) \big)   =    \epsilon\big( s^\ell(\epsilon (u))  s^\ell(a) \big)
= \epsilon(u)  a  =   s^r_{\scalastd}(a)(u).
\end{split}
\end{equation*}
Likewise, the second identity follows from
\begin{equation*}
\begin{split}
S^{\scalast}\big(t_r^{\scalast}(a)\big)(u)   &=  \epsilon\big( u_+  t^\ell\big(t_r^{\scalast}(a)(u_-) \big) \big)
=   \epsilon\big( u_+  t^\ell(a  \epsilon (u_-)) \big)
%\\
% &=   \epsilon\big( u_+  t^\ell(\epsilon(u_-))  t^\ell(a) \big)
= \epsilon(ut^\ell(a))
=   t^r_{\scalastd}(a)(u).
\end{split}
\end{equation*}
The last identity in {\em (a)} regarding the respective counits is for $\phi \in U^{\scalast}$ proven by the line
$$
\pl_{\scalastd} S^{\scalast}(\phi) = S^{\scalast}(\phi)(1_\uhhu) = \phi(1_\uhhu) = \partial^{\scalast} \phi.
$$
As for  {\em (b)},  let us first more generally compute an element $S^\scalast(\phi )\psi$ for $\phi \in U^\scalast$ and $\psi \in U_\scalastd$: by  \cite[Eq.~(3.1.1)]{Kow:HAATCT}, Eq.~\rmref{Sch4}, and the properties of a bialgebroid counit, we have
\begin{equation*}
\begin{split}
\langle S^\scalast(\phi )\psi ,u \rangle &=
\langle \psi, t^\ell(\langle u_{(2)} , S^\scalast(\phi)\rangle)u_{(1)}\rangle
=\langle \psi, t^\ell(\langle \epsilon, u_{(2)+}t^\ell(\langle\phi, u_{(2)-}\rangle)\rangle)u_{(1)} \rangle \\
&= \langle \psi, t^\ell(\langle \epsilon, u_{+(2)}t^\ell(\langle\phi, u_{-}\rangle)\rangle)u_{+(1)}\rangle \\
&=\langle \psi, t^\ell(\langle \epsilon, u_{+(2)}s^\ell(\langle\phi, u_{-}\rangle)\rangle)u_{+(1)} \rangle \\
&= \langle \psi, t^\ell(\langle \epsilon, u_{+(2)}\rangle) u_{+(1)}t^\ell(\langle\phi, u_{-}\rangle)\rangle
=\langle\psi, u_{+}t^\ell(\langle\phi, u_{-}\rangle)\rangle.
\end{split}
\end{equation*}
With the help of this property, by  \cite[Eq.~(3.1.2)]{Kow:HAATCT}
along with  \rmref{Sch5}, \rmref{Sch9}, and the fact that the counit in $U$ gives the unit in $U_\scalastd$, one sees that for all  $  \phi, \phi' \in U^{\scalast}  $
 \begin{equation*}
 \begin{split}
 \langle S^{\scalast}( \phi\phi') , u \rangle
 &= \langle  \epsilon , u_+ t^\ell(\langle \phi \phi', u_- \rangle )\rangle
= \langle  \epsilon , u_+ t^\ell(\langle \phi', s^\ell\phi(u_{-(1)}) u_{-(2)} \rangle )\rangle \\
 &= \langle  \epsilon , u_{++} t^\ell(\langle \phi', s^\ell\phi(u_{-}) u_{+-} \rangle )\rangle \\
&= \langle  \epsilon , (u_{+} t^\ell\phi(u_-))_+ t^\ell(\langle \phi', (u_{+} t^\ell\phi(u_-))_-
 \rangle )\rangle \\
 &= \langle  S^\scalast(\phi') \epsilon , u_{+} t^\ell\phi(u_-)\rangle =  \langle  S^{\scalast}(\phi)S^{\scalast}( \phi' ), u \rangle.
 \end{split}
 \end{equation*}
Observe that if $\due U \lact {}$ is finitely generated $A$-projective, then {\em (b)} follows by the fact that \rmref{diqua} defines an action, but in general we do not want to assume this at this point.

For proving {\em (c)} --- when $U_\ract$ and $\due U \lact {}$ are finitely generated $A$-projective ---,
one could equally do this by a straightforward somewhat technical computation. A quicker way is to use the results
in \S\ref{fallouts}:
denoting the right coproduct on $U_{\scalastd}$ resp.\ $U^{\scalast}$ by Sweedler superscripts,
one has
\begin{equation*}
\begin{split}
S^\scalast(\phi)^{(1)} \otimes_\ahha S^\scalast(\phi)^{(2)}
&=
(\epsilon \otimes_\ahha \epsilon) S^\scalast(\phi)
=   (\epsilon \otimes_\ahha \epsilon) \Yleft \phi  \\
&= (\epsilon \Yleft \phi^{(1)}) \otimes_\ahha (\epsilon \Yleft \phi^{(2)})
=  S^\scalast(\phi^{(1)}) \otimes_\ahha S^\scalast(\phi^{(2)}),
\end{split}
\end{equation*}
where in the first equation we used the monoidal structure on  $\mathbf{Mod}\mbox{-}U_{\scalast}$, and in the third the fact that all functors in \rmref{eastpak} are strict monoidal.

   The second part in {\em (i)} --- about the  $ U $-linearity  of  $ S^{\scalast} $ ---,   which is straightforward, is left to the reader.
\end{proof}

\begin{remark}
\label{S* as "antipode"}
 When  $ U $  is just a Hopf algebra over  $  A = k  $  with antipode  $ S $,  we have  $  U^{\scalast} = (U_\scalastd)^\op_\coop  $, and $ S^{\scalast} $  is nothing but the transpose of $ S  $.  If  $ U^{\scalast}$  itself is in turn a Hopf algebra --- namely, if the transpose of the multiplication  $ m_\uhhu $  in  $ U $  takes values in the tensor square of  $ U^{\scalast} $   ---,   then  $ S^{\scalast} $  is just the antipode of this dual Hopf algebra  $ U^{\scalast}  $.  In this context,  Theorem \ref{S*-morphism} simply expresses the fact that the antipode in a Hopf algebra is an antimorphism of algebras and of coalgebras.
\end{remark}

In particular, in case $U$ is both a left and right Hopf algebroid we have:

\begin{theorem}
\label{criptabalbi}
  Let  $ (U,A) $  be simultaneously a left and a right Hopf algebroid.
Then the maps  $ S^{\scalast} $  and  $ S_{\scalastd} $  are inverse to each other. Hence, if both  $ A $-modules  $ \due U \lact {} $  and  $ U_\ract $  are, in addition, finitely generated projective, $ (S^{\scalast}, \id_\ahha ) $  and  $ (S_{\scalastd}, \id_\ahha) $  are isomorphisms of right bialgebroids which are inverse to each other.
\end{theorem}

\begin{proof}
 As for the first statement, we directly compute  by means of the bialgebroid axioms along with  \rmref{mampf3}  and  \rmref{Tch7}, for any  $  \phi \in U^{\scalast}  $:
 \begin{equation*}
 \begin{split}
 (S_{\scalastd} S^{\scalast} \phi)(u) &= \epsilon\big(u_{[+]} s^\ell(S^{\scalast} \phi(u_{[-]}))\big)
= \epsilon\big(u_{[+]} s^\ell\big(\epsilon_U(u_{[-]+} t^\ell \phi(u_{[-]-}))\big)\big) \\
 &= \epsilon\big(u_{[+]}u_{[-]+} t^\ell \phi(u_{[-]-})\big)
=  \epsilon\big(u_{(2)[+]}u_{(2)[-]}t^\ell \phi(u_{(1)})\big) \\
 &= \phi(u_{(1)})\epsilon(u_{(2)}) = \phi(u),
 \end{split}
 \end{equation*}
which proves that  $  S_{\scalastd} \circ S^{\scalast} = \id_{\uhhu^{\scalast}}  $.   Likewise, one shows that  $  S^{\scalast} \circ S_{\scalastd} = \id_{\uhhu_{\scalastd}}  $.
%
%   The last statement about the  $ (U \otimes U) $-linearity  now follows from the previous facts and from  Theorem \ref{S*-morphism}.
%
\end{proof}

\subsection{The case of a full Hopf algebroid}
\label{napoleon}

If  $ H $  is a full Hopf algebroid with bijective antipode $S$ in the sense of \cite{BoeSzl:HAWBAAIAD}, then it is, in particular, both a left and right bialgebroid (see  the short summary below): therefore --- still assuming that  $ \due H \lact {} $  and  $ H_\ract $  are both finitely generated projective as  $ A $-modules ---, there is a right bialgebroid analogue to the previous constructions concerning the maps  $ S^{\scalast} $  and  $ S_{\scalastd} $.  On the other hand, the antipode $ S $  induces by transposition new maps  $ S^t $,  $\qttrd S t {}{}{}{} $,  etc., for the dual spaces.  Hereafter we discuss links between these various maps, in particular showing that, while for the Hopf algebra case one has identities like  $  S^{\scalast} = \qttrd S t {}{}{}  $  ({\em cf.}~Remark \ref{S* as "antipode"}),  this is no longer the case for the general setup of full Hopf algebroids as illustrated in \S\ref{enoteca}  below.

\begin{free text}{\bf Reminder on full Hopf algebroids.}
\label{reminders_H-ads}
 Recall that a full Hopf algebroid structure (see, for example, \cite{Boe:HA}) on a  $ k $-module  $ H $
consists of the following data:
\begin{enumerate}
\compactlist{99}
\item
a left bialgebroid structure  $  H^\ell := ( H, A, s^\ell, t^\ell, \Delta_\ell  , \epsilon)  $
over a  $ k $-algebra  $ A  $;
\item
 a right bialgebroid structure  $  H^r := ( H, B, s^r, t^r, \Delta_r  , \partial)  $ over a  $ k $-algebra  $ B $;
\item
the assumption that the  $ k $-algebra  structures for  $ H $  in  {\em (i)\/}  and in  {\em (ii)\/} be the same;
\item
a $ k $-module  map  $  S : H \to H  $;
\item
some compatibility relations between the previously listed data for which we refer to {\em op.\ cit.}
\end{enumerate}
We shall denote by lower Sweedler indices the left coproduct  $ \Delta_\ell $  and by upper indices the right coproduct  $ \Delta_r  $, that is,
$  \Delta_\ell(h) =: h_{(1)} \otimes_\ahha h_{(2)}  $  and  $  \Delta_r(h) =: h^{(1)} \otimes_\behhe h^{(2)}  $
 for any  $ h \in H  $.
As said before, a full Hopf algebroid (with bijective antipode) is both a left and right Hopf algebroid but not necessarily vice versa (as illustrated in \S\ref{enoteca}).
In this case, the translation maps in \rmref{latoconvalida} are given by
\begin{equation}
\label{laterza}
h_+ \otimes_\Aopp h_- = h^{(1)} \otimes_\Aopp S(h^{(2)}) \quad \mbox{and} \quad h_{[+]} \otimes_\Bopp h_{[-]} = h^{(2)} \otimes_\Bopp S^{-1}(h^{(1)}),
\end{equation}
formally similar as for Hopf algebras.
\end{free text}
The following lemma \cite{Boe:HA, BoeSzl:HAWBAAIAD} will be needed to prove the main result in this subsection.

\begin{lemma}
\label{vitasnella}
  Let  $ H $  be any Hopf algebroid.  Then
 \begin{enumerate}
 \compactlist{99}
\item

the maps $ \nu  :=  \partial s^\ell : A \to B^\op  $  and  $  \mu  :=  \epsilon s^r : B \to A^\op  $  are isomorphisms of  $ k $-algebras;
 \item
 the pair of maps $(S, \nu) : H^\ell \to {(H^r)}^\op_\coop  $  gives an isomorphism of left bialgebroids;
 \item
 the pair of maps
  $ (S, \mu) : H^r \to {(H^\ell)}^\op_\coop  $  gives an isomorphism of right bialgebroids.
 \end{enumerate}
 \end{lemma}

The next observation might let us consider  $ S^{\scalast} $  and  $ S_{\scalastd} $  as sort of an analogue of the antipode on the dual:

\begin{proposition}  \label{dual-cocomm_Hopf}
%\label{U^*=U_*}
 Let  $ (U, A)  $  be a cocommutative left bialgebroid (in particular, $A$ is commutative and $s^\ell = t^\ell$).
Then  $(U,A) $  is a  left Hopf  algebroid if and only if it is a right Hopf algebroid; in this case, assuming in addition that  $ \due U \lact {} $  and  $ U_\ract $  are finitely generated $A$-projective,
% one has that  $  S^{\scalast} = S_{\scalastd}  $  and  %%% was better before ...
$  (U^{\scalast},A) = ((U_{\scalastd})_\coop,A)  $  is a full Hopf algebroid with involutive  antipode  $  \mathscr{S} := S^{\scalast} =  S_{\scalastd}  $.
\end{proposition}

\begin{proof}
The first claim directly holds true by the very definitions.  The rest of the proof follows  {\em verbatim\/}  in the footsteps of the one of  Theorem 3.17  in  \cite{KowPos:TCTOHA},  which considers the special case  for $  U = V^\ell(L)$.
\end{proof}

As mentioned before, one can also link the duals of a Hopf algebroid $(H,S)$ by transposed maps  $ {}^{t\!}S $, which usually do not coincide with $S^{\scalast}$ or $S_{\scalastd}$ (see also \S\ref{enoteca}). The next result explains
a relation
between them.

\begin{theorem}
\label{movimentoallaforza}
 Let  $ H $  be a Hopf algebroid such that  $ \due H \lact {} $  and  $ H_\ract $  are finitely generated $A$-projective.  Then the diagram
\begin{equation*}
  \xymatrix{  {\big( {(H^r)}^\op_\coop \big)}^{\scalast}  \ar@{->}^{{}^{t\!}S}[rr]  \ar@{->}_{S^{\scalast}_r}[d]  &  & {(H^\ell)}^{\scalast}  \ar@{->}^{S^{\scalast}_\ell}[d]  \\
  {\big( {(H^r)}^\op_\coop \big)}_{\scalastd}  \ar@{->}_{{}^{t\!}S}[rr]  &  &  {(H^\ell)}_{\scalastd}  }
\end{equation*}
of right bialgebroid morphisms
is commutative.
\end{theorem}

\begin{proof}
 Let us identify  $ B^\op $  and  $ A $  by means of the $ k $-algebra  isomorphism  $  \nu : A \to B^\op  $ mentioned above;  then the left algebroid  $  {(H^r)}_\coop^\op $  is described by the sextuple
  $$  \big( {(H^r)}^\op  ,  \widehat{s^\ell} := s^r \nu  ,  \widehat{t^\ell} := t^r \nu  , \Delta_r^\coop  ,  \widehat{\epsilon} := {\nu}^{-1}  \partial \big)  .  $$
Moreover, the Hopf algebroid  $  \big( {(H^r)}_\coop^\op  , {(H^\ell)}_\coop^\op  , (S, \mu) : {(H^r)}_\coop^\op \to H^\ell \big)  $  is the one we have to consider to compute  $ S_r^{\scalast}  $.  For  $  \phi \in {\big( {(H_r)}^\op_\coop \big)}_{\!\scalast}  $  and  $  h \in H  $   we have

\begin{equation*}
\begin{split}
 \langle( {}^{t\!}S \circ S_r^{\scalast})(\phi)  , h \rangle
 &=   \widehat{\epsilon} \big( {S(h)}_{(2)}  \widehat{t^\ell}\big(\langle \phi  , S( {S(h)}_{(1)} ) \rangle \big)
=   \big( \nu^{-1} \partial S \big) \big( h^{(1)}  t^\ell\big( \langle \phi  , S^2( h^{(2)})\rangle \big) \\
&=   \epsilon\big( h^{(1)}  t^\ell\big(\langle \phi  , S^2( h^{(2)})\rangle\big) \big) \\
&=   \epsilon\big( h^{(1)}  t^\ell\big(\langle {}^{t\!}S(\phi)  , S( h^{(2)})\rangle\big) \big)
=   \langle ( S_\ell^{\scalast} \circ {}^{t\!}S)(\phi)  , h \rangle,
\end{split}
\end{equation*}
 where we used the explicit form \rmref{laterza} of the translation map and the fact
that $S$ is an anti-coring morphism between left and right coproduct, which proves $  {}^{t\!}S \circ S_r^{\scalast}  =  S_\ell^{\scalast} \circ {}^{t\!}S  $   as claimed.
\end{proof}

\begin{remark}
\label{schonschoen}
 In general, both maps  $ S^{\scalast} $  or  $ S_{\scalastd} $  can be thought of as an extension of the notion of antipode for a full Hopf algebroid, in the following sense.  As mentioned in Lemma \ref{vitasnella}, the antipode in a full Hopf algebroid $H$ yields a bialgebroid morphism $  S : H^\ell \to {(H^r)}^\op_\coop  $. On the other hand, if  $ U $  is a left Hopf algebroid, for which  $ \due U \lact {} $  and  $ U_\ract $  are finitely generated projective as  $ A $-modules,  then we have a similar situation replacing  $( H^\ell , H^r  , S)  $   with the triple  $ ((U^{\scalast})^\op , (U_{\scalastd})_\coop  , S^{\scalast} )$, and one might be tempted to define a Hopf algebroid as a triple $(U,V,S)$ of a left resp.\ right bialgebroid $U$ resp.\ $V$, where the underlying ring structure is {\em not\/}  the same: this way, the apparent asymmetry of a Hopf algebroid consisting of two coring structures but only one ring structure (that makes it difficult to obtain self-duality) would be somewhat attenuated.
On the other hand, in case a left Hopf algebroid is simultaneously a right Hopf algebroid, by Theorem \ref{criptabalbi} both duals are isomorphic and hence can be seen (under the stated finiteness conditions) as {\em its} dual (right) bialgebroid, which carries a Hopf structure by the results in \cite{Schau:TDATDOAHAAHA}.
\end{remark}

\section{Examples and applications}
\label{lautesGequatsche}
In this section we present some further developments and some applications
to specific examples.

\subsection{Mixed distributive law between duals}
   A direct application of the existence of the bialgebroid morphism  $ S^{\scalast} $  (or  $ S_{\scalastd} $)  is to the setup of  {\it distributive laws}.  Indeed, a particular kind of mixed distributive law (or entwining) in the sense of Beck \cite{Bec:DL} can be constructed via the following recipe.  Combining a morphism  $(\phi_1, \phi_0) : (V, B)  \to (V', B')  $  of right (say) bialgebroids with a Hopf-Galois map yields
  $$  \chi : \due {V'} {} \bract \otimes^\behhe  \due V \lact {} \to \due V {} \bract \otimes_\behhe  \due {V'} \blact {},  \quad  v' \otimes^\behhe v  \mapsto  v^{(1)} \otimes_\behhe v'  \phi(v^{(2)}),
$$
which can be easily seen to define a mixed distributive law between $  V'  $  (thought of as a coalgebra) and  $V$  (thought of as an algebra, although its coproduct appears in $\chi$).
Applying this to the two duals of a left bialgebroid $U$ along with $S^{\scalast}$, one obtains
  $$  \chi : U_{\scalastd} {}_\bract \otimes^\ahha  \due {U^{\scalast}} \lact {} \to \due {U^{\scalast}} {} \bract  \otimes_\ahha  \due {U_{\scalastd}} \blact {},  \qquad  \psi \otimes^\ahha \phi
\mapsto  \phi^{(1)} \otimes_\ahha \psi  S^{\scalast}(\phi^{(2)})
$$
as a mixed distributive law between  $ U^{\scalast} $  and  $ U_{\scalastd}  $, to which
any standard construction based on mixed distributive laws could be applied.

\subsection{Lie-Rinehart algebras and their jet spaces}
\label{enoteca}
Let  $(A, L) $  be a Lie-Rinehart algebra ({\em cf.}~\cite{Rin:DFOGCA}, geometrically a Lie algebroid).
Then its (left) universal enveloping algebra  $ V^\ell(L) $ carries not only the structure of a left bialgebroid over the commutative algebra $A$ (see \cite{Xu:QG}) but also of a left Hopf algebroid \cite{KowKra:DAPIACT}:  on generators $a \in A$ and $X \in L$, its translation map is given by
\begin{equation}
\label{transl-map_V(L)}
  a_+ \otimes_\Aopp a_-  =  a \otimes_\Aopp 1   ,  \quad  X_+ \otimes_\Aopp X_-  =  X \otimes_\Aopp 1 - 1 \otimes_\Aopp X.
\end{equation}
Moreover, as  $ V^\ell(L) $  is cocommutative, it is also a right Hopf algebroid.

  {\em Full\/}  Hopf algebroid structures on  $ V^\ell(L) $  are in bijection with right  $ V^\ell(L) $-module  structures  on  $ A  $ which play the r\^ole of possible right counits,  expressed by suitable maps
 $  \partial : V^\ell(L) \to A  $  ({\em cf.}~\cite[\S4.2]{Kow:HAATCT} or \cite{KowPos:TCTOHA} for more information). The corresponding antipode $  S : V^\ell(L) \to V^\ell(L)^\op_\coop  $  is then uniquely determined by the prescriptions
\begin{equation}
\label{militeignoto}
S(a)  =  a   ,  \quad  S(X) = -X + \partial(X),  \qquad \forall  a \in A  , \  \forall  X \in L,
\end{equation}
 on generators. For a general Lie-Rinehart algebra (which does not arise from a Lie algebroid), such a map $\pl$ and hence the antipode might or might not exist.

%   As $A$ is commutative, one can think of  $ L $  as being a  {\em right\/}  Lie-Rinehart algebra in the sense specified in \cite[\S2.1.8]{CheGav:DFFQG}. The difference this makes here only appears on the level of the universal enveloping object: one obtains a different algebra  $ V^r(L) $,  which in turn is a right bialgebroid, see {\em op.~cit.}: one has $V^r(L)=V^\ell(L^\op)^\op$, where $L^\op$ is the Lie Rinehart algebra opposite to $L$ obtained from $L$ by taking the opposite bracket and the opposite anchor.

   Let us consider the (right)  {\em jet spaces}  $  J^r(L) := V^\ell(L)^{\scalast}  $  and
$  {}^r\!J(L) := V^\ell(L)_{\scalastd}  $.  If  $ L $  is finitely generated projective as an  $ A $-module,  then  $ J^r(L) $  and  $  {}^r\!J(L) $  are right bialgebroids  {\em in a suitable topological sense},  as their coproduct takes values in a  {\em topological\/}  tensor product; concerning this, we quickly recall some non-trivial key facts, referring to  \cite{KowPos:TCTOHA, CalVdB:HCAC}  for further details.

First, $ V^\ell(L) $  is the direct limit of an increasing bialgebroid filtration ({\em i.e.}, the strict analogue of a bialgebra filtration) of finitely generated projective modules  $ V^\ell(L)_n  $;  it follows that  $  J^r(L)$
in turn is the inverse limit of all the  $  J^r(L)_n := (V^\ell(L)_n)^{\scalast}  $,  which are finitely generated projective as well.  Similar remarks apply to  $  {}^r\!J(L)  $.  As  $  V^\ell(L)_p \cdot V^\ell(L)_q \subseteq V^\ell(L)_{p+q}  $  (for all  $ p, q \in \N  $),  the recipe used to define the coproduct in  $ U^\scalast $  when  $ U $  is a left bialgebroid such
 that  $ U_\ract $  is
finitely generated $A$-projective (see  \S \ref{regnetswirklich?}) can be applied again and yields maps
  $$
J^r(L)_n = (V^\ell(L)_n)^{\scalast} {\buildrel {\Delta^{J^r}_{\,n}} \over {\relbar\joinrel\longrightarrow}} \hskip-5pt {\textstyle \sum\limits_{p+q=n}} \hskip-5pt
(V^\ell(L)_p)^{\scalast} \!
 {}_{\bract}  \otimes_\ahha {}_{\blact}
 (V^\ell(L)_q)^{\scalast}
  \hskip-1pt  = \hskip-6pt {\textstyle \sum\limits_{p+q=n}} \hskip-5pt  J^r(L)_p {}_{\bract}  \otimes_\ahha  {}_{\blact} J^r(L)_q
$$
whose inverse limit  $  \Delta^{J^r} \! := \lim\limits_{\leftarrow\joinrel\relbar} \Delta^{J^r}_{\,n}  $  is the coproduct of  $ J^r(L)  $.  Similarly, one constructs ``coproduct-like maps''  $ \Delta^{{}^r\!J}_{\,n} $  for the  $  {}^r\!J(L)_n := (V^\ell(L)_n)_{\scalastd}  $  and then takes their inverse limit  $  \Delta^{{}^r\!J} := \lim\limits_{\leftarrow\joinrel\relbar} \Delta^{{}^r\!J}_{\,n}  $  as a coproduct for  $  {}^r\!J(L)  $.

   Now, because of the very definition of the  $ V^\ell(L)_n $  and of the explicit form  \rmref{transl-map_V(L)}  of the translation map of  $ V^\ell(L)  $,  one easily finds that the translation map itself (much like the coproduct) maps every  $  V^\ell(L)_n  $  into  $  \sum_{p+q=n} V^\ell(L)_p \otimes_\Aopp  V^\ell(L)_q  $.  Then formula \rmref{sstarup}  makes sense again, and thus can be used to produce a well-defined map
  $$
S^{\scalast}_n : J^r(L)_n = (V^\ell(L)_n)^{\scalast} \relbar\joinrel\relbar\joinrel\longrightarrow {\big(V^\ell(L)_n\big)}_{\scalast} = {}^r\!J(L)_n.
$$
Moreover, the arguments used in the proof of  Theorem \ref{S*-morphism}  to show that $ S^\scalast $  preserves the coproduct apply again in the present situation, and yield a commutative diagram
\begin{equation}
\label{diagr_Sstar-Delta}
\begin{gathered}
  \xymatrix{
    J^r(L)_n  \ar@{->}[d]_{S_n^{\scalast}} \ar[rr]^{\Delta^{J^r}_{\,n} \hskip33pt}  &  &  {\textstyle \sum\limits_{p+q=n}} \hskip-3pt  J^r(L)_p \,
 {}_{\bract} \otimes_\ahha {}_{\blact}
J^r(L)_q  \ar@{->}[d]^{{\sum\limits_{p+q=n}} \hskip-1pt S_p^{\scalast} \otimes S_q^{\scalast}}  \\
     {}^r\!J(L)_n  \ar[rr]_{\Delta^{{}^r\!J}_{\,n} \hskip33pt}  &  &  \hskip-1pt  {\textstyle \sum\limits_{p+q=n}} \hskip-3pt  {}^r\!J(L)_p \,
 {}_{\bract} \otimes_\ahha {}_{\blact}
{}^r\!J(L)_q  }
\end{gathered}
\end{equation}

\noindent
 Taking the inverse limit of all these  $ S^{\scalast}_n $  we get a well-defined (continuous) map
  $$
S^{\scalast} : J^r(L) = V^\ell(L)^{\scalast} \relbar\joinrel\relbar\joinrel\longrightarrow V^\ell(L)_{\scalast} = {}^r\!J(L).
$$
It follows by construction that this map necessarily coincides with the same name map in  \S\ref{linking},  hence it respects all  $ \Ae $-ring  structure maps of  $ J^r(L) $  and  $ {}^r\!J(L) $  as well as their counits; from  \rmref{diagr_Sstar-Delta}  follows that this map also respects the coproduct on both sides.  All in all, this means that  $ S^{\scalast} $  is a morphism of (topological) bialgebroids.
As  $ V^\ell(L) $  is also a right Hopf algebroid, \S\ref{linking}  also provides a map
  $
S_{\scalastd} : {}^r\!J(L) \to J^r(L),
$
which again turns out to be a morphism of (topological) bialgebroids, inverse to $ S^{\scalast} $.  The outcome is that
 \vskip1pt
   \centerline{ \it Theorem \ref{S*-morphism} holds true (in full strength) for  $  U = V^\ell(L) $ }
 \vskip1pt
\noindent
(replacing the formulation ``morphism of right bialgebroids'' by ``morphism of topological right bialgebroids''),
although the left bialgebroid  $ V^\ell(L) $  does not comply with the finiteness assumptions required (in general) for that result.

   Finally, note that both  $ J^r(L) $  and  $ {}^r\!J(L) $  are commutative (because  $ V^\ell(L) $  is cocommutative), so they are also left bialgebroids.  Identifying $  J^r(L) $ as the coopposite of
$  {}^r\!J(L) $ and with the cocommutativity of  $ V^\ell(L)  $,  one finds that  $ S^{\scalast} $  and  $ S_{\scalastd} $  are equal and yield an  {\em antipode\/}  for  $ J^r(L)  $,  which in this way becomes a  full Hopf algebroid.
  In other words,  Proposition \ref{dual-cocomm_Hopf}  holds true for  $  U = V^\ell(L)  $  and  $  U^{\scalast} = J^r(L) = {}^r\!J(L)_\coop = (U_{\scalastd})_\coop  $,  although  $  V^\ell(L)  $  is  {\em not\/}  finitely generated projective.

\subsubsection{Difference between $S^*$ and $ {}^{t\!}S$}

In this specific example, one can explicitly observe the difference between $S^{\scalast}$ and the transpose of the antipode $S$ on $V^\ell(L)$ in \rmref{militeignoto}.
Apart from the fact mentioned above that $S^{\scalast}$ always exists while ${}^{t\!}S$ does not, this is already clear on an abstract level since these are maps of different nature as pointed out in  Theorem \ref{movimentoallaforza}. Nevertheless, one directly sees here that with respect to the $A$-module structures coming from left and right multiplication in $V^\ell(L)$, the map $S^\scalast(\phi)$ is left $A$-linear whereas $ {}^{t\!}S(\phi)$ is $A$-linear from the right, for $\phi \in  V^\ell(L)^{\scalast}$.
Evaluating both maps on an element in $L \subset V^\ell(L)$, one obtains
\begin{equation*}
\label{tSxVL}
   {}^{t\!}S(\phi)(X) = -  \phi(X) + \partial(X) \phi(1)  \qquad \qquad  \forall  \phi \in V^\ell(L)^{\scalast},  X \in L,
\end{equation*}
on one hand, and on the other hand:
$$
S^{\scalast}(\phi)(X)   =  - \phi(X) + X(\phi(1)) \qquad \qquad  \forall  \phi \in V^\ell(L)^{\scalast}  ,  X \in L,
$$
where $L \to \Der(A,A), \ X \mapsto \{ a \mapsto X(a)\}$ denotes the anchor of the Lie-Rinehart algebra $(A,L)$. Using the property $Xa - aX = X(a)$ with respect to the product in $V^\ell(L)$ as well as the right $A$-linearity of $\partial$, one obtains $\pl(aX) = \pl(X) a - X(a)$ and therefore
$
{}^{t\!}S(\phi)(X) - S^{\scalast}(\phi)(X) = \partial(\phi(1)X),
$
which in general does not vanish.

\subsection{Examples from quantisation}

  In this section, we adapt our main constructions and results to a different setup, that of quantisations of universal enveloping algebras (of Lie-Rinehart algebras) and other associated objects.  In particular, this means that we deal with yet another kind of topological bialgebroids, so that we have to clarify the nature of these objects and how the analysis and results of
the preceding sections
%\S\S \ref{acquaprimavera}--\ref{cater_fallouts}
fits to this modified context.

\begin{definition}  \label{def-q-bialgd}
\label{def-QUEAd}
 Let  $ \big( U, A, s^\ell, t^\ell, m, \Delta , \epsilon \big) $  be a left (resp.\ right) bialgebroid.
A {\em quantisation of $U$} (or {\em quantum bialgebroid\/}) is a  {\em topological\/}   left (resp.\ right) bialgebroid  $ \big( U_h, A_h , s^\ell_h , t^\ell_h , m_h , \Delta_h , \epsilon_h \big) $  over a topological  $ k[[h]] $-algebra  $ A_h $  such that:
\begin{enumerate}
\compactlist{99}
\item
$  A_h $  is isomorphic to  $ A[[h]] $  as a topological  $ k[[h]] $-module,  and this isomorphism induces an algebra isomorphism  $  A_h \big/  h  A_h  \cong  A[[h]] \big/  hA[[h]]  \cong  A  $;
\item
     $  U_h $  is isomorphic to  $ U[[h]] $  as a topological  $ k[[h]] $-module;
\item
   $  U_h \big/  hU_h  \cong U[[h]] \big/  hU[[h]]  $  is isomorphic to  $U $  as a left  $ A $-bialgebroid  via the isomorphism  $  A_h \big/  hA_h \cong A[[h]] \big/  hA[[h]] \cong A  $    mentioned in (i);
\item
the coproduct  $ \Delta_h $  of  $ U_h $  takes values in  $  U_h  \widehat{\times}_{\ahha_h} U_h  $, where
 \begin{equation*}
\qqquad   {U_h} \widehat{\times}_{\ahha_h} {U_h}  :=
 \big\{ {\textstyle \sum_i} u_ i \otimes u'_i \in
{U_h}_\ract  \, \widehat{\otimes}_{\ahha_h} \due {U_h} \lact {}
 \mid {\textstyle \sum_i}  (a \blact u_i) \otimes u'_i = {\textstyle \sum_i} u_i \otimes (u'_i \blacktriangleleft a) \big\}
\end{equation*}
is the  {\em Takeuchi-Sweedler product},  and where
$
 {U_h}_\ract  \, \widehat{\otimes}_{\ahha_h} \due {U_h} \lact {}
$
denotes the completion of
$
 {U_h}_\ract  \, {\otimes}_{\ahha_h} \due {U_h} \lact {}
$
with respect to the  $ h $-adic  topology.
\end{enumerate}
   In this setting, we shall say that  $ U_h $  is a  {\em quantisation},  or  {\em quantum deformation},  of  $ U  $.
\end{definition}

\begin{remark}  \quad
\label{rem.s_quant-def.s}

   {\it (a)}\, The notions of quantum left or right Hopf algebroid are defined replacing the ordinary tensor product by a suitable completion, just as for $ J^r(L) $ above.

   {\it (b)}\,  When dealing with  $ k[[h]] $-modules,  any morphism ({\em i.e.},  $ k[[h]] $-linear  map) is automatically continuous for the  $ h $-adic  topology on the source and target  $ k[[h]] $-module; we shall tacitly use this fact with no further mention.  In particular, for a quantum bialgebroid  $ U_h $ both its (full linear) duals  $ (U_h)^{\scalast} $  and  $ (U_h)_{\scalastd} $  are also  {\em topological duals}.

   {\it (c)}\,
For a left bialgebroid $ U $  with a quantisation $ U_h $, assume that  $ U $  is also a left Hopf algebroid. Then  $ U_h $  is automatically a left Hopf algebroid (in a topological sense) as well by a standard argument in deformation theory:
by assumption, we have  $  U_h \cong U[[h]]  $  as modules over  $  A_h \cong A[[h]]  $;  from this isomorphism one deduces similar isomorphisms for modules of homomorphisms or tensor products of modules.  Moreover --- because  $  U_h \Big/ h U_h \cong U  $  as bialgebroids ---,   all bialgebroid structure maps of  $ U_h $  taken modulo  $ h $  reduce to the same name structure maps of  $ U  $.  Now, for the (topological) left bialgebroid  $ U_h $  we have a well-defined Hopf-Galois map
  $$
(\ga_\ell)_h : \due {U_h} \blact {} \widehat{\otimes}_{{\scriptscriptstyle{A^{\rm op}_h}}} \, {U_h}_\ract  \to {U_h}_\ract  \, \widehat{\otimes}_{\ahha_h} \, \due {U_h} \lact {}  ,  \quad
   u \, \widehat{\otimes}_{A^{\rm op}_h} v \mapsto u_{(1)} \, \widehat{\otimes}_{\ahha_h}  \, u_{(2)} v,
$$
which belongs to
$  \Hom_{k[[h]]}\big( \due {U_h} \blact {} \widehat{\otimes}_{{\scriptscriptstyle{A^{\rm op}_h}}} {U_h}_\ract  ,  {U_h}_\ract \widehat{\otimes}_{\ahha_h} \due {U_h} \lact {} \big)  $:  as mentioned above, this module is isomorphic to
  $  \Hom_k\big( \due U \blact {} \otimes_{\Aopp} U_\ract,  U_\ract  \otimes_\ahha  \due U \lact {} \big)[[h]]  $,  so that  $ {(\ga_\ell)}_h $  expands as  $  (\ga_\ell)_h = \sum_{n \in \N} a_n h^n  $  for some  $  a_n \in \Hom_k\big( \due U \blact {} \otimes_{\Aopp} U_\ract  ,  U_\ract  \otimes_\ahha \due U \lact {} \big)  $.
In addition, as all structure maps of  $ U_h $  modulo  $ h $  are just those of  $ U  $,  one has  $  \ga_\ell = (\ga_\ell)_h \mod h = a_0  $.  But  $  U $  was a left Hopf algebroid, hence  $  \ga_\ell = a_0  $  is invertible, and therefore  $  (\ga_\ell)_h = \sum_{n \in \N} a_n h^n  $  is invertible too, so that  $  U_h $  is a left Hopf algebroid as well.
\end{remark}

\begin{free text}{\bf Universal enveloping algebras and deformations.}
 As in  \cite{CheGav:DFFQG},  one can consider a quantum deformation  $ V^\ell(L)_h $  of  $ V^\ell(L) $:  as the latter is both a left and right Hopf algebroid, the same holds true for  $ V^\ell(L)_h $  as well, by Remark \ref{rem.s_quant-def.s} {\it (c)\/}  above.

   On the other hand, the dual (right) bialgebroids  $  J^r(L)_h := (V^\ell(L)_h)^{\scalast}  $  and
$  {}^r{\!}J(L)_h = (V^\ell(L)_h)_{\scalastd}  $  are deformations of  $  J^r(L) = V^\ell(L)^{\scalast} = (V^\ell(L)_{\scalastd})_\coop  $.  This common ``limit'' is a full Hopf algebroid (with bijective antipode) by the above, hence in particular it is a left and right Hopf algebroid with respect to the underlying right bialgebroid structure. It then follows that the same is true for the right bialgebroids  $ J^r(L)_h $  and  $ {}^r{\!}J(L)_h  $, but usually they are not full Hopf algebroids.  Nonetheless, we can apply our constructions of  \S\ref{linking}  to  $  U_h := V^\ell(L)_h  $  and find the maps  $ S^{\scalast} $  and  $ S_{\scalastd}  $,  as we now shortly explain.

   By construction, the maps  $ S^{\scalast} $  and  $ S_{\scalastd} $  as in  (\ref{sstarup})  and  (\ref{sstardown})  are given in terms of structure maps and translation maps of the (non-topological) bialgebroid  $ U  $:  when  $ U $  is replaced by  $ U_h $,  all those maps are continuous, hence both definitions still make sense and provide maps
$  S^{\scalast} : (U_h)^{\scalast} \to (U_h)_{\scalastd}  $  and  $  S_{\scalastd} : (U_h)_{\scalastd}  \to (U_h)^{\scalast}  $  as announced.  Once these maps are properly defined (for  $  U_h = V^\ell(L)_h  $),  the proof of all their properties still works untouched (all arguments and calculations make sense and go through in the proper setup of topological bialgebroids).  In particular,  Theorem \ref{criptabalbi}  then assures that the two deformations  $  J^r(L)_h := (U_h)^{\scalast}  $  and  $  {}^r{\!}J(L)_h := (U_h)_{\scalastd}  $  of  $  V^\ell(L)^{\scalast} = (V^\ell(L)_{\scalastd})_\coop  $  are isomorphic (as right bialgebroids) via  $ S^{\scalast} $  and  $ S_{\scalastd}  $.
\end{free text}

\subsection{Cases where a dualising module exists}
\label{heisz}
 In this section, we will come back to the situation of dualising modules as in \S\ref{malgenauerhinsehen} by investigating their (deformation) quantisation. To this end, we first need to introduce some extra notation, terminology, and definitions with respect to decreasing filtrations; see, for example, \cite{Che:DPFQG, Schn:AITDM} for further basic results and details.

%   To begin with, we introduce some extra notation and terminology.
Let $A$ be an algebra endowed with a decreasing filtration $(F_nA)_{n \in \N}$ and consider a filtered $FA$-module denoted by $FM$, whereas  its underlying $A$-module will be denoted by $M$.  If  $FM$ and $FN$ are two filtered $FA$-modules, then
a filtered morphism $Fu : FM \to FN$ is a morphism $u :M \to N$ of the underlying $A$-modules such that $u(F_sM) \subset F_s N$.  A filtered morphism  $Fu : FM \to FN$ is {\em strict} if it satisfies
$u(F_{s}M)=u(M)\cap F_{s}N$.  An exact sequence of $FA$-modules is a sequence
\begin{equation}
\label{effervescentenaturale}
FM \buildrel{Fu}\over \longrightarrow FN
\buildrel{Fv} \over \longrightarrow FP
\end{equation}
such that $\Ker(F_s v) = \mathrm{Im}\,(F_s u)  $,  where  $  F_s v := v{\big|}_{F_s N}  $  and  $  F_s u := u{\big|}_{F_s M}  $; hence  $ Fu $  is strict.  If moreover  $ Fv $  is also strict, then \rmref{effervescentenaturale}  is a called a {\em strict exact sequence}.

   The filtration of a filtered module gives rise to a topology and even a metric if the filtered module is separated, that is, if $\bigcap_{n \in \N} F_nM=\{0\}  $.
For any $r \in \Z$ and for any $FA$-module $FM$, we define the
{\em shifted module $FM(r)$} as the module $M$ endowed with the filtration
$( F_{s+r}M )_{s \in \Z}$.
An $FA$-module is called {\em finite free} if isomorphic to an
$FA$-module of the type
$\bigoplus_{i=1}^{p}FA(-d_{i})$, where
$d_{1},\dots , d_{p} \in\Z$. An $FA$-module $FM$ is called {\em of finite
type} if one  can find
$m_{1} \in F_{d_{1}}M, \dots, m_{p}\in F_{d_{p}}M$ such that any
$m \in F_{d}M$ may be written as
$$
m={\displaystyle \sum_{i=1}^{p}a_{d-d_{i}}}m_{i},
$$
where $a_{d-d_{i}} \in F_{d-d_{i}}A$.
We will be dealing with the case where $M$ is a $k[[h]]$-module and
$F_nM = h^n M$, the so-called {\em $h$-adic filtration}.

\begin{remark}
 The existence of a translation map if $U_h$ is a left or right Hopf algebroid makes it possible to endow
\begin{itemize}
\item[--]
$\Hom$-spaces with values in a $h$-adic complete space, and
\item[--]
complete tensor products of $U_h$-modules
\end{itemize}
with further natural $U_h$-module structures.  Let us make this explicit for the cases we will use, {\em i.e.}, adapt Proposition \ref{structures}.

   If ${\mathcal P}_h$ is a right $U_h$-module and $N_h$ is a left $U_h$-module, then
$_{\blact}{\mathcal P}_h \otimes_{{\scriptscriptstyle{A^{\rm op}_h}}} {N_h}_{\ract}$ is endowed with a right $U_h$-module structure as follows:
if $u \in U_h$, then
$u_+ \otimes_{{\scriptscriptstyle{A^{\rm op}_h}}} u_{-} \in {}_\blact U_h \, \widehat{\otimes}_{{\scriptscriptstyle{A^{\rm op}_h}}} {U_h}_\ract$
can be written as
$u_+\otimes_{{\scriptscriptstyle{A^{\rm op}_h}}} u_- = \lim\limits_{n \to \infty} u_{+,n}\otimes_{{\scriptscriptstyle{A^{\rm op}_h}}} u_{-,n}$.
For
$x\otimes_{{\scriptscriptstyle{A^{\rm op}_h}}} y \in
{}_{\blact}{\mathcal P}_h\otimes_{{\scriptscriptstyle{A^{\rm op}_h}}} {N_h}_{\ract}$,
one defines
$$
(x\otimes_{k[[h]]}y) u :=  \lim\limits_{n \to \infty}  x u_{+,n}\otimes_{{\scriptscriptstyle{A^{\rm op}_h}}} u_{-,n} y
\in
{}_{\blact}{\mathcal P}_h \otimes_{{\scriptscriptstyle{A^{\rm op}_h}}} {N_h}_{\ract}.
$$
As
$
\lim\limits_{n \to \infty} t^\ell(a)u_{+,n}\otimes_{{\scriptscriptstyle{A^{\rm op}_h}}}u_{-,n} = \lim\limits_{n \to \infty} u_{+,n}\otimes_{{\scriptscriptstyle{A^{\rm op}_h}}}u_{-,n}t^\ell(a)
$,
 we have thus defined a right action of $U_h$ on
$_{\blact}{\mathcal P}_h\otimes_{{\scriptscriptstyle{A^{\rm op}_h}}} {N_h}_{\ract}$.

   If ${\mathcal P}_h$ and $N_h$ are right  $U_h$-modules, then
$ \Hom_{{{\scriptscriptstyle{A^{\rm op}_h}}}}({\mathcal P}_h, N_h)$ is endowed with a left $U_h$-module structure as follows: if
$u_{[+]}\otimes^{\ahha_h} u_{[-]} =\lim\limits_{n \to \infty} u_{[+],n} \otimes^{\ahha_h} u_{[-],n}
\in U_{h \bract} \, \widehat{\otimes}^{A_h} {}_\lact U_h$,
one sets for
$  \phi \in \Hom_{A_h^{\text{op}}}({\mathcal P}_h, N_h)  $  and  $  u \in U _h  , \,  p \in P_h  $,
$$
u_n \phi (p) :=\phi \left ( p u_{[+],n}\right )u_{[-],n},
$$
and argues similarly as above that this defines, indeed, a left $U_h$-action on $\Hom_{A_h^{\text{op}}}({\mathcal P}_h, N_h)$.
\end{remark}

\begin{lemma}
 Let $(U_h,A_h)$
be a quantum left Hopf algebroid, and let ${\mathcal P}_h$ be a right $U_h$-module such that ${\mathcal P}_{h \blacktriangleleft}$
(respectively  $_{\blact}{\mathcal P}_{h}$)
is a finitely generated projective $A_h^{{\rm op}}$-module (resp.~$A_h$-module). Then
\begin{enumerate}
\compactlist{99}
\item
${\mathcal P}_h$ is complete for the $h$-adic topology.
\item
For a right $U_h$-module $N_h$, any element of $\Hom_{{{\scriptscriptstyle{A^{\rm op}_h}}}}({\mathcal P}_h, N_h)$ is continuous if we endow both modules with the $h$-adic topology.
\item
If $N_h$ is a left $U_h$-module that is  complete in the $h$-adic topology, then so is the right $U_h$-module
${}_{\blact}{\mathcal P}_h\otimes_{{\scriptscriptstyle{A^{\rm op}_h}}}{N_h}_{\ract}$.
\item
 If $N_h$ is a right $U_h$-module that is  complete in the $h$-adic topology, then so is the left $U_h$-module
$\Hom_{{{\scriptscriptstyle{A^{\rm op}_h}}}}({\mathcal P}_h, N_h)$.
\end{enumerate}
\end{lemma}

\begin{proof}
If $N_h$ is a right $U_h$-module endowed with the $h$-adic topology, then
the $h$-adic topology on $(N_h)^p$ coincides with the product topology. Thus, if $N_h$ is complete for the $h$-adic topology, then so is $(N_h)^p$.
\begin{enumerate}
\compactlist{99}
\item
As $\cP_h$  is a finitely generated projective $A_h^{{\rm op}}$-module, it is a summand of a free module, which is complete for the $h$-adic topology as $A_h$ is so. Hence $\cP_h$ is complete for the $h$-adic topology.
\item
This is obvious as such a morphism is $k[[h]]$-linear.
\item
$\cP_h$ is a direct summand of a rank $r$ free $A_h^{{\rm op}}$-module $F_h$.
Thus
${}_{\blact}{\mathcal P}_h \otimes_{{\scriptscriptstyle{A^{\rm op}_h}}} N_{\ract}$ is a summand of $(N_h)^r$, which
is complete, hence it is itself complete.
\item
The proof of this part is analogous to the proof of (iii).
\end{enumerate}
\end{proof}

In the following, denote by $\cmodu$ resp.\ $\ucmod$  the category of right resp.\ left $U_h$-modules which are
complete for the $h$-adic topology. We then have the following result, analogous to Proposition \ref{left and right modules}:

\begin{proposition}
\label{left and right modules h}
 Let  $(U_h,A_h)$  be simultaneously a quantum left and right Hopf algebroid.
Assume that there exists a right  $ U_h $-module  $ \cP_h$, where $ {\cP_h}_{\bract} $
(resp.\ ${}_\blact \cP_h$)  is finitely generated projective over $A_h^{{\rm op}}$ (resp.\  $A_h$), such that
\begin{enumerate}
\compactlist{99}
\item
the left $U_h$-module morphism
 $$
A_h \to \Hom_{{\scriptscriptstyle{A^{\rm op}_h}}}(\cP_h  , \cP_h),  \quad  a  \mapsto  \{ p \mapsto a \blact p \}
$$
 is an isomorphism of $k[[h]]$-modules;
\item
the evaluation map
\begin{equation*}
%\label{lacartachenontagliaglialberi}
\due {\cP_h} \blact {} \otimes_{{\scriptscriptstyle{A^{\rm op}_h}}} \Hom_{{\scriptscriptstyle{A^{\rm op}_h}}}(\cP_h  , N_h)_{\ract} \to N_h,
\quad p  \otimes_{{\scriptscriptstyle{A^{\rm op}_h}}}  \phi  \mapsto  \phi(p)
\end{equation*}
is an isomorphism for any $ N_h \in \cmodu $.
\end{enumerate}
   Then
$$
\ucmod \to \cmodu, \quad M_h \mapsto \due {\cP_h} \blact {} \otimes_\Aopp {M_h}_\ract
$$
is an equivalence of categories with quasi inverse given by
$  N'_h  \mapsto  \Hom_{{\scriptscriptstyle{A^{\rm op}_h}}}(\cP_h , N'_h)  $.
\end{proposition}

   We will now give an example of such a situation. Consider a left bialgebroid $(U,A)$ and a quantisation $(U_h, A_h)$ of it.  Observe that the natural left $U_h$-module structure on $A_h$ quantises that of $U$ on $A$.

\begin{theorem}
\label{maximal exterior power}
Let $(U,A)$ be a left bialgebroid, where $U$ is assumed to be a $k$-Noetherian algebra. Assume that  there exists an integer $d$ satisfying
$$
\Ext^i_U(A, U)=
\left \{
\begin{array}{ll}
0 &  \mbox{if}  \ i \neq d,\\
\Lambda & \mbox{if}  \ i=d.
\end{array}
 \right .
$$
%Here, $U$ acts from the right on $\Lambda$ by right multiplication.
Then there exists an $A_h$-module $\Lambda_h$ that is a quantisation of $\Lambda$ such that
$$
\Ext^i_{U_h}(A_h, U_h)=
\left \{
\begin{array}{ll}
0 & \mbox{if}  \ i \neq d,\\
\Lambda_h & \mbox{if} \ i = d,
\end{array}
 \right.
$$
where the right action of $U_h$ on
 $\Ext^d_{U_h}(A_h, U_h)$ is a quantisation of the right action of $U$ on $\Ext^d_{U}(A, U)$ given by right multiplication.
 \end{theorem}

We remind the reader here that $\Lambda_h$ is $\Lambda[[h]]$ as a $k[[h]]$-module.
This theorem is proven in \cite{Che:DPFQG} in the case where $A_h=k[[h]]$.
For the proof of the general case, we will need the following auxiliary statement:

\begin{lemma}
\label{resolution}
There exists  a resolution of the $U_{h}$-module $A_h$ by
 finite rank free (filtered) $FU_{h}$-modules
$$
\ldots \buildrel {\partial_{i+1}} \over \longrightarrow FL^{i}
\buildrel{\partial_{i}}\over \longrightarrow  \ldots \buildrel{\partial_{2}} \over
\longrightarrow   FL^{1} \buildrel{\partial_{1}} \over
\longrightarrow FL^{0} \longrightarrow A_h \longrightarrow \{0\},
$$
where $FL^{i}$ is $(U_h)^{d_i}$ endowed with the $h$-adic filtration
such that the associated graded complex
$$
\ldots GL^{i}\buildrel {G\partial_{i}} \over \longrightarrow
 \ldots \to GL^{1}\buildrel {G\partial_{1}} \over \longrightarrow
  GL^{0}\longrightarrow A[h] \longrightarrow \{0\}
$$
is a resolution of the $U[h]$-module $A[h]$.
\end{lemma}

\begin{proof}
We will construct the $p$-th module $FL^p$ by induction on $p$: for $p=0$, one may take $FL^0:=U_h$ and $\partial _0:=\epsilon$, endowed with the $h$-adic topology. Assume then that $FL^0, FL^1, \ldots, FL^p$ are already constructed along with $\partial_0, \partial_1, \ldots,  \partial_p$. As $FL^p$ is topologically free, the induced filtration and the  $h$-adic filtration coincide on
$\Ker \partial_p$. As $\Ker \partial_p$ is closed in  $ FL^p $, it is also complete.
This $k[[h]]$-module is topologically free as it is complete for the $h$-adic topology and also torsion free; set
$\Ker \partial_p :=V_p[[h]]$.
Since $GU_h=U[h]$ is Noetherian, the (filtered) algebra $U_h$ is (filtered) Noetherian \cite[Prop.~3.0.7]{Che:DPFQG} and the
$U_h$-module $\Ker \partial_p$ is finitely generated so that the $U$-module $V_p$ is finitely generated as well. Let
$(\overline{v_1}, \ldots , \overline{v_{d_{p+1}}})$ be a generating system of the $U$-module $V_p$ and let
$(v_1, \ldots , v_{d_{p+1}})\in (\Ker \partial_p)^{d_{p+1}}$
be a lift of $(\overline{v_1}, \dots , \overline{v_{d_{p+1}}})$.  Moreover, introduce the $U_h$-module morphism
$$
%\begin{array}{rrcl}
\partial_{p+1}: (U_h)^{d_{p+1}} \to  \Ker \partial_p, \quad (u_1,\ldots ,u_{p+1}) \mapsto  \sum u_i v_i,
%\end{array}
$$
which is a strict
morphism of filtered modules. The filtered exact sequence
$$
(U_h)^{p+1}\buildrel {\partial_{p+1}} \over \longrightarrow (U_h)^p
\buildrel{\partial_{p}}\over \longrightarrow
(U_h)^{p-1}
$$
is strict exact so that the sequence
$$
(GU_h)^{p+1}\buildrel {G\partial_{p+1}} \over \longrightarrow (GU_h)^p
\buildrel{G\partial_{p}} \over
\longrightarrow   (GU_h)^{p-1}
$$
is exact ({\em cf.}~\cite[Prop.~3.0.2]{Che:DPFQG}).
\end{proof}

\begin{proof}[Proof of Theorem \ref{maximal exterior power}]
The $\Ext ^\bull_{U_h}(A_h, U_h)$-groups
can be computed via
the complex $M^\bull := \big( \Hom_{U_{h}}(L^{\bullet}, U_h), \partial_{\bullet} \big)$.
Its components are endowed with the natural filtration
$$
F_s\Hom_{U_{h}}(L^{i}, U_{h}) :=
\{\lambda \in \Hom_{U_{h}}(L^{i}, U_{h})
\mid \lambda ( F_{p}L^{i} )\subset
F_{s+p}U_{h}\},
$$
and the  right $FA$-modules
$F\Hom_{U_{h}}(L^{i}, U_{h})$ are isomorphic to
$(U_h)^{d_i} $ endowed with the $h$-adic  filtration.
On the other hand,
the filtration of the $M^i:= \Hom_{U_{h}}(L^{i}, U_{h})$ induces a filtration on
$\Ext^i_{U_h}(A_h, U_h)$ as follows:
$$
F_s\Ext^i_{U_h}(A_h, U_h) :=
\frac{\Ker  \qttr \partial t {}{} i \cap F_s M^{i} + {\rm Im}  \qttr \partial t {}{} {i-1}}{{\rm Im}  \qttr \partial t {}{} {i-1}}
\simeq
\frac{\Ker \qttr \partial t {}{} i\cap F_s M^{i}}{{\rm Im}  \qttr \partial t {}{} {i-1}\cap
    F_s M^{i-1}}.
$$
The filtration on the
$\Ext^i_{U_h}(A_h, U_h)$-groups is nothing but the $h$-adic filtration.
Reproducing the proof of \cite{Che:DPFQG}, one can see that:
\begin{itemize}
\item
if $i \neq d$, then
$\Ext^{i}_{U_{h}} (A_h, U_{h})
=\{0\}$;
\item the maps $\qttr \partial t {}{} i$ are strict filtered morphisms;
\item
$\Ext^{d}_{U_{h}} (A_h, U_h )$ is complete for the $h$-adic filtration (as it is a finitely generated $U_h^\op$-module, see \cite{Che:DPFQG}).  Moreover,
$\Ext^{d}_{U_{h}}(A_h, U_{h})/h\Ext^{d}_{U_{h}}(A_h,U_{h})  \simeq \Ext^{d}_{U}(A,U)$ as $U^\op$-modules.
\end{itemize}

Let us show that $\Ext^{d}_{U_{h}} (A_h, U_{h} )$ is $h$-torsion free.
Let $[\sigma_d] \in \Ext^{d}_{U_{h}} (A_h, U_{h} )$, where $\sigma_d \in \Ker \qttr \partial t {}{} d$, be an $h$-torsion element in $\Ext^{d}_{U_{h}} (A_h, U_{h})$. There exists a minimal $n \in \mathbb{N}^{\scalast}$ such that
$h^n[\sigma_d]=0$.
Let $\sigma_{d-1} \in \Hom _{U_h}(L^{d-1}, U_h)$ be such that
$h^n \sigma_d = \qttr \partial t {}{} {d-1}(\sigma_{d-1})$. Then, by reduction modulo $h$, one obtains
$\overline{ \qttr \partial t {}{} {d-1}}(\overline{\sigma_{d-1}})=0$ and there exists $\overline{\sigma_{d-2}}$ such that
$\overline{\sigma_{d-1}}=\overline{\partial_{d-2}}\left (\overline{\sigma_{d-2}} \right )$. Let $\sigma_{d-2}$ be a lift of $\overline{\sigma_{d-2}}$. Then there exists $\tau_{d-1}$ such that
$$
\sigma_{d-1}= \qttr \partial t {}{} {d-2} (\sigma_{d-2}) + h \tau_{d-1}.
$$
Hence $h^n\sigma_{d}=h \qttr \partial t {}{} {d-1} (\tau_{d-1})$, which gives
(using the fact that $\Hom_{U_{h}}(L^{d}, U_{h}) $ is topologically free)
$h^{n-1}\sigma_{d}= \qttr  \partial t {}{} {d-1} (\tau_{d-1})$.
This contradicts the minimality of $n$ so that
$\Ext^{d}_{U_{h}} (A_h, U_{h})$ is $h$-torsion free.
As $\Ext^{d}_{U_{h}}(A_h, U_{h})$ is complete for the $h$-adic topology and $h$-torsion free,  it is topologically free.
\end{proof}

Combining this result with the more general structure theory as in Proposition \ref{left and right modules} resp.\ Proposition \ref{left and right modules h}, one obtains:

 \begin{proposition}
\label{tavernadeiquaranta}
 Let $U$ satisfy the conditions of Theorem \ref{maximal exterior power}. Assume moreover that
 \begin{enumerate}
 \item
 $A$ is noetherian;
 \item
  $\Ext_U(A,U)$ is a dualising module for $(U,A)$, {\em i.e.}, satisfies the hypothesis of Proposition
 \ref{left and right modules};
% \item
% $\Ext_U(A,U)_\bract$  is a finitely generated projective $A^\op$-module;
 \item
 $_\blact \Ext_U(A,U)$ is a  finitely generated projective  $A$-module.
 \end{enumerate}
 Then
 $\cP_h= \Ext_{U_h}^d ( A_h, U_h)$ is a dualising module for $(U_h,A_h)$ and produces
an equivalence between the categories of left resp.\ right complete $U_h$-modules.
 \end{proposition}

\begin{remark}
\label{piovera?}
Let  $M_h :=M[[h]]$ and $N_h := N[[h]]$ be two   $A_h^{\op}$-modules which are  topologically free with respect to the $h$-adic topology.  Assume moreover that $M_h$ is finitely generated projective over $A^\op_h$;
then $\Hom_{\scriptscriptstyle{A_h^{\op}}}(M_h, N_h)$ is topologically free and, as said before,  is isomorphic to
 $\Hom_{\Aopp}(M,N)[[h]]$ as a $k[[h]]$-module:
 observe that $\Hom_{\scriptscriptstyle{A_h^{\op}}}(M_h, N_h)$ is complete for the induced  topology  as it is a closed subset of the topologically free $k[[h]]$-module
 $\Hom_{k[[h]]}(M_h, N_h)$. On the other hand, on  $\Hom_{\scriptscriptstyle{A_h^{\op}}}(M_h, N_h)$, the induced topology coincides with the $h$-adic topology. Hence $\Hom_{\scriptscriptstyle{A_h^{\op}}}(M_h, N_h)$ is complete
 for the $h$-adic topology and
 since it is also torsion free, it is topologically free. Let us now show that
 $\Hom_{\scriptscriptstyle{A_h^{\op}}}(M_h, N_h)/h \Hom_{\scriptscriptstyle{A_h^{\op}}}(M_h, N_h)$ is isomorphic to
$\Hom_{A^{\op}}(M, N)$: in fact, there exists an $A_h^{\op}$-module $M_h^\prime$ and a finitely generated free  $A_h^{\op}$-module $F_h$
such that
 $M_h \oplus M_h^\prime =F_h$. Any element $\phi$ of
 $\Hom_\Aopp(M,N)$ can be extended to an element of
$\Hom_\Aopp({F_h/hF_h},N)$, which, in turn, can be lifted to an element of
$\Hom_{\scriptscriptstyle{A_h^{\op}}}(F_h, N_h)$ and produces (by restriction) a lift of $\phi$.
\end{remark}

\begin{proof}[Proof of Proposition \ref{tavernadeiquaranta}]
The module ${\cP_h}_\bract$ is
a finitely generated $A_h^\op$-module as $\cP_\bract := \Ext_U(A,U)_\bract$ is a finitely generated $\Aop$-module
(see Proposition 3.0.5 of the preprint version of \cite{Che:DPFQG}).

   Let $N_h$ be a finitely generated $A_h^\op$-module.  It can be considered as a filtered $FA_h^\op$-module as follows: one has an epimorphism
$  {\big( A_h^{\op} \big)}^n \, {\buildrel {p} \over \longrightarrow} \, N_h \longrightarrow 0  $,
and we endow $N_h$ with the filtration $p\big( F{\big( A_h^{\op} \big)}^n \big)$.
 As ${\cP}_\bract$ is a projective $A^\op$-module, ${\cP}[h]_\bract$ is a projective $A[h]^\op$-module, and
Proposition 3.0.11 of the preprint version of \cite{Che:DPFQG}  shows that
$\Ext^i_{A_h^\op}(\cP_h,N_h)=\{0\}$ if $i>0$.

   Let now $N_h$ be any $A_h^\op$-module. We have
$N_h =\lim\limits_{\rightarrow} N'_h$, where $N'_h$ runs over all finitely
generated  $ A_h^\op $-submodules  of  $ N_h  $. Let $F^{\bullet}$ be a resolution of $\cP$ by
finitely generated free $A_h^{\op}$-modules. We have
\begin{equation*}
\begin{split}
\Ext^{j}_{A_h^\op}(\cP_h,N_h) &= \Ext^{j}_{A_h^\op}(\cP_h,\lim\limits_{\rightarrow}N'_h)
= H^{j} \big( \Hom_{A_h^\op}(F^{\bullet},\lim\limits_{\rightarrow}N'_h)\big)\\
&= H^{j}\big( \lim\limits_{\rightarrow}\Hom_{A_h^\op}(F^{\bullet}, N'_h)\big)
=  \lim\limits_{\rightarrow} H^{j} \big( \Hom_{A_h^\op}(F^{\bullet}, N'_h)\big)\\
&= \lim\limits_{\rightarrow} \Ext^{j}_{A_h^\op}(\cP_h, N'_h)
= \{0\},
\end{split}
\end{equation*}
where we used the fact that
the functor $\lim\limits_{\rightarrow}$ is exact because the set of
finitely generated submodules of $M$ is a directed set, {\em cf.}~\cite[Prop.~5.33]{Rot:AITHA}.
Thus we have proven that if $N_h$ is any $A_h^\op$-module, then
$$
\Ext^{j}_{A_h^\op} (\cP_h,N_h)=\{0\} \quad {\rm if } \  j>0.
$$
Consequently, ${\cP_{h}}_\bract$ is a projective $A^\op_h$-module; similarly, ${}_\blact \Ext_{U_h}(A_h,U_h)$ is a projective $A_h^\op$-module.

   The assertion with respect to the evaluation map yet is true if $N_h$ is a topologically free $U_h$-module as it is true  modulo $h$, see Remark \ref{piovera?}.
Furthermore, the functor  $  N_h \mapsto \cP_h \otimes_{\scriptscriptstyle{A_h}}\Hom_{\scriptscriptstyle{A^\op_h}}(\cP_h, N_h)  $  is exact as
$\cP_{h \bract}  $  resp.\  $_\blact \cP_h$  is a projective $A_h^\op$-module resp.\ $A_h$-module.

Let now $N_h$ be a finitely generated $U_h$-module. Using a finite free resolution of  $ N_h  $,  one can show (by a diagram chase argument) that the evaluation map is an isomorphism (as it is an isomorphism for any component of the resolution).
If $N_h$ is any $U_h$-module instead, one can write  $  N_h = \lim\limits_{\rightarrow} N'_h  $, where $N'_h$ runs over all finitely
generated submodules of  $ N_h  $.  Since ${\cP}_h$ is a finitely generated $A_h^{\op}$-module, any element
$\phi \in \Hom_{\scriptscriptstyle{A^\op_h}}(\cP_h, N_h)$ can be considered
as an element of $\Hom_{\scriptscriptstyle{A^\op_h}}(\cP_h, N'_h)$ for a well-chosen finitely generated $A_h^{\op}$-module  $ N'_h  $. Using the finitely generated case, one can see that
 the  evaluation map is an isomorphism for any $U_h$-module  $ N_h  $.

As $\cP_h$ is a finitely generated projective $A_h^{\op}$-module,
the natural left $U_h$-module map
$$
A_h \to \Hom_{\scriptscriptstyle{A^\op_h}}(\cP_h, \cP_h),
\quad a \mapsto (p \mapsto a \blact p)
$$
of Proposition \ref{left and right modules h}
 is an isomorphism as it is an isomorphism modulo $h$. This concludes the proof.
\end{proof}

\begin{example}
For example, if $A$ is the algebra of regular functions on a smooth affine variety $X$ and $L$ is the Lie-Rinehart algebra of vector fields over  $ X  $,  then  $  U = V^\ell(L)  $  satisfies the conditions of Theorem \ref{maximal exterior power}.  More generally, for any Lie-Rinehart algebra $(A,L)$, where $L$ is finitely generated projective of constant rank $d$ over a Noetherian algebra $A$, the pair $\big(A, V^\ell(L)\big)$ fulfils the conditions of  Theorem \ref{maximal exterior power} and one obtains
 $\Ext^d_{V^\ell(L)}(A,V^\ell(L)) = \bigwedge_\ahha^d \Hom_\ahha(L,A)$ for the dualising module (see \cite{Che:PDFKALS, Hue:DFLRAATMC} for more details in this direction).
%
%%%
% Under the same conditions, also the pair $\big(A, J^r(L)\big)$ fits into Theorem
% \ref{maximal exterior power} with $\Ext^d_{J^r(L)}(A,J^r(L)) = \bigwedge_\ahha^d L$.
%%%
%
 Then, for any quantisation  $ V^\ell(L)_h $  of  $ V^\ell(L)  $,
%
%%%
% resp.\ $J^r(L)_h$ of $J^r(L)$
%%%
%
 Proposition \ref{tavernadeiquaranta} leads to an equivalence of categories between left and right
complete $V^\ell(L)_h$-modules.
%
%%%
% and likewise for left and right $J^r(L)_h$-modules.
%%%
%
 Examples of quantisations of  $ V^\ell(L) $  are given in \cite{CheGav:DFFQG}.
\end{example}

\end{document}